\documentclass[12pt, reqno]{amsart}
\usepackage{amsmath, amsthm, amscd, amsfonts, amssymb, graphicx, color}
\usepackage[bookmarksnumbered, colorlinks, plainpages]{hyperref}
\hypersetup{colorlinks=true,linkcolor=red, anchorcolor=green, citecolor=cyan, urlcolor=red, filecolor=magenta, pdftoolbar=true}
\input{mathrsfs.sty}
\usepackage{xcolor}
\usepackage{tikz}
\usetikzlibrary{decorations.pathreplacing,angles,quotes}
\definecolor{cof}{RGB}{219,144,71}
\definecolor{pur}{RGB}{186,146,162}
\definecolor{greeo}{RGB}{91,173,69}
\definecolor{greet}{RGB}{52,111,72}
\usepackage{xcolor}
\usepackage{ulem}
\newtheorem{theorem}{Theorem}[section]
\newtheorem{lemma}[theorem]{Lemma}
\newtheorem{proposition}[theorem]{Proposition}
\newtheorem{cor}[theorem]{Corollary}
\theoremstyle{definition}
\newtheorem{definition}[theorem]{Definition}
\newtheorem{example}[theorem]{Example}

\theoremstyle{remark}
\newtheorem{remark}[theorem]{\bf{Remark}}
\numberwithin{equation}{section}
\allowdisplaybreaks
\begin{document}
	
	\title[On approximate preservation of orthogonality]{On approximate preservation of orthogonality and its application to isometries}
	
	\author[Mandal, Manna, Paul and Sain  ]{Kalidas Mandal,  Jayanta Manna,  Kallol Paul and Debmalya Sain }

	\address[Mandal]{Department of Mathematics\\ Jadavpur University\\ Kolkata 700032\\ West Bengal\\ INDIA}
\email{iamjayantamanna1@gmail.com}

	\address[Manna]{Department of Mathematics\\ Jadavpur University\\ Kolkata 700032\\ West Bengal\\ INDIA}
	\email{iamjayantamanna1@gmail.com}

	\address[Paul]{Vice-Chancellor, Kalyani University \& Professor of Mathematics\\ Jadavpur University (on lien) \\ Kolkata \\ West Bengal\\ INDIA}
	\email{kalloldada@gmail.com}
	
		\newcommand{\acr}{\newline\indent}
	\address[Sain]{Department of Mathematics\\ Indian Institute of Information Technology, Raichur\\ Karnataka 584135 \\INDIA}
	\email{saindebmalya@gmail.com}
	
	
	\begin{abstract}
		Motivated by the famous Blanco-Koldobsky-Turn\v{s}ek characterization of isometries, we study the  \textit{approximate preservation of Birkhoff-James orthogonality by a linear operator between  Banach spaces}. In particular, we investigate various geometric and analytic properties related to such preservation on finite-dimensional polyhedral Banach spaces. As an application of the results obtained here, we present refinements of the  Blanco-Koldobsky-Turn\v{s}ek characterization of isometries on certain Banach spaces. 
	\end{abstract}

	\maketitle
	
	\section{Introduction.}
	Determining the isometries on a Banach space is of fundamental importance in functional analysis and operator theory. One of the classical  results in the isometric theory of Banach spaces is the following characterization of isometries due to Koldobsky, Blanco and Turn\v{s}ek \cite{BT06,K93}:\\
	
	\textit{Every bounded linear operator on a Banach space preserving Birkhoff-James orthogonality is an isometry multiplied by a constant.}\\
	
	Recently,  a \textit{local} study of the preservation of Birkhoff-James orthogonality was initiated in \cite{SMP24},  highlighting its importance in understanding  isometries in terms of the geometries of the concerned spaces. A continuation of this idea was pursued in \cite{MMPS25}, to obtain a finer characterization of isometries, in terms of \textit{directional preservation of orthogonality.}  The purpose of this article is to introduce \textit{an approximate version of  preservation of orthogonality by a linear operator between  Banach spaces}. As we will see, this leads to further refinements of the Blanco-Koldobsky-Turn\v{s}ek characterization of isometries in case of certain Banach spaces. Let us now introduce the relevant notations and terminologies  to be used in this article.\\
	
	The letters  $ \mathbb{X}, \mathbb{Y} $ denote Banach spaces.  Throughout this article, we will be working with \textit{real} Banach spaces of dimension strictly greater than $1.$ The dual space of $\mathbb{X}$ is denoted by $\mathbb{X}^*.$ Let
	$B_{\mathbb{X}} = \{x\in \mathbb{X} : \|x\|\leq 1\}$ and $S_{\mathbb{X}} = \{x\in \mathbb{X} : \|x\|=1\}$
	be the  unit ball and the unit sphere of $\mathbb{X}$, respectively. Let $\mathbb{L}(\mathbb{X}, \mathbb{Y})$ denote the space of all bounded linear operators from $ \mathbb{X} $ to $ \mathbb{Y}, $ endowed with the usual operator norm.  Whenever $ \mathbb{X} = \mathbb{Y}, $ we simply write $ \mathbb{L}(\mathbb{X}, \mathbb{Y}) = \mathbb{L}(\mathbb{X}).$   For a nonempty convex set $A \subset \mathbb{X},$ $ x \in A$  is said to be an  extreme point of $A,$ if $x =(1-t)y +tz,$ for some $ t \in(0,1)$ and some $y, z\in A,$ implies that $x = y = z.$ The set of all extreme points of $A$ is denoted by $Ext~A.$ A convex set $F\subset S_{\mathbb{X}}$ is said to be a face of $B_{\mathbb{X}}$ if for every $x_1,x_2\in S_{\mathbb{X}},~ (1-t)x_1+tx_2\in F$ implies that $x_1,x_2\in F,$ where $0<t<1.$ A maximal face of $B_{\mathbb{X}}$ is said to be a facet of $B_{\mathbb{X}}$.  The convex hull of a nonempty set $A\subset \mathbb{X}$ is denoted by $co(A).$ The conical hull of a nonempty set $A\subset \mathbb{X}$ is given by the set $\big\{\sum\limits_{i=1}^{n}a_ix_i : x_i\in A \text{ and }a_i\geq0 \big\}.$ A subset $K$ of $\mathbb{X}$ is said to be a convex cone if 
	\[\text{(i)}~K+K\subset K\text{ and (ii)}~\alpha K\subset K, \text{ for all } \alpha > 0.\] For a nonempty convex set $C \subset \mathbb{X},$ $\text{Int}_r~C$ denotes the relative interior of the set $C,$ i.e., $\text{Int}_r~C=\{x\in C : \text{ there exists }\epsilon>0\text{ such that } B(x,\epsilon)\cap \text{affine}(C)\subset C\},$ where $\text{affine}(C)$ is the intersection of all affine spaces containing $C.$ Let us recall that an affine space is defined as a translation of a subspace of the vector space. A finite-dimensional Banach space $\mathbb{X}$ is said to be polyhedral if $Ext~B_{\mathbb{X}}$ is finite. Clearly, the space  $ \mathbb{X} $ is polyhedral if and only if $ \mathbb{X}^* $ is polyhedral.  Note that for a finite-dimensional polyhedral Banach space $\mathbb{X}$,  $f \in Ext~B_{\mathbb{X}^*} $ if and only if $f$ is an extreme  supporting functional corresponding to a facet of $B_{\mathbb{X}}$ \cite[Lemma 2.1]{SPBB19}.
    
	For a non-zero element $x \in \mathbb{X},$ we denote the collection of all supporting functionals at $x$ by $J(x),$ i.e., $ J(x) = \{ f \in S_{\mathbb{X}^*} : f(x) = \| x \| \}. $  A non-zero  $ x \in \mathbb{X} $ is said to be a smooth point if there exists a unique supporting functional at $x,$ i.e., if $J(x)$ is a singleton. For any $k \in \mathbb{N}, $ a non-zero  $ x \in \mathbb{X} $ is said to be a $ k $-smooth point (or, that the order of smoothness of $ x $ is $k$)  if $dim~span \,J(x)=k.$  For more information on smoothness and $k$-smoothness, the readers may see \cite{DMP22,KS05,LR07,MDP22,MP20,PSG16,SPMR20,W18} and the references therein. 
	
	Given any two elements $x, y \in \mathbb{X}, $ we say that $x$ is Birkhoff–James orthogonal \cite{B35, J47}  to $y,$ written as $ x \perp_B y,$ if $ \| x + \lambda y \| \geq \|x\|$ for all scalars $\lambda. $ It follows from the Hahn-Banach theorem that for a given $x \in \mathbb{X},$ there exists non-zero $y \in \mathbb{X}$ such that $ x \perp_B y.$ The set of all such $y$ is denoted by $ x^{\perp_B} $, i.e., $x^{\perp_B} = \{ y \in \mathbb{X} : x \perp_B y \}. $ For more information on Birkhoff-James orthogonality and its applications in studying operator spaces, we refer the readers to the recently published monograph \cite{Book24}. 
	The following definition \cite{C05} presents an approximate version of Birkhoff-James orthogonality, which plays an important role in describing the geometry of Banach spaces \cite{CSW17,RS23,W22}. 
	\begin{definition}\label{approx}
		Let $\mathbb {X}$  be a normed  linear space and let $x,y\in \mathbb{X}.$ Let $\epsilon\in[0,1).$ Then $x$ is said to be $\epsilon$-approximate Birkhoff-James orthogonal to $y,$ written as $x\perp_B^{\epsilon}y,$ if 
		\[\|x+\lambda y\|^2\geq \|x\|^2-2\epsilon\|x\|\|\lambda y\|,\text{ for all scalars } \lambda.\]
	\end{definition}
	
	 For the sake of simplicity, we will call $\epsilon$-approximate Birkhoff-James orthogonality as $\epsilon$-orthogonality throughout this article. The $\epsilon$-orthogonal set of  $x\in\mathbb{X}$ is denoted by $ x^{\perp_B^\epsilon} $ and is defined by  $x^{\perp_B^\epsilon} = \{ y \in \mathbb{X} : x \perp_B^\epsilon y \}. $ An operator $T\in\mathbb{L}(\mathbb{X},\mathbb {Y})$ is said to preserve  $\epsilon$-orthogonality at $x\in\mathbb{X},$ if for any  $y\in \mathbb{X}, ~ x\perp_B y\implies Tx\perp_B^{\epsilon} Ty.$\\  
	
	In view of the Blanco-Koldobsky-Turn\v{s}ek characterization of isometries, it is natural to consider the following more general query:\\
	
	\textit{Given a Banach space $ \mathbb{X}, $ does there exist an $\epsilon \in (0, 1),$  such that any $ T \in \mathbb{L}(\mathbb{X}) $ preserving  $\epsilon$-orthogonality  is an isometry multiplied by a constant?}\\
	
	It is clear that an affirmative answer to the above query will directly lead to a proper refinement of the Blanco-Koldobsky-Turn\v{s}ek theorem, which is the main motivation behind this study. 
    In order to address the above query, it is convenient to introduce the following definition:
	\begin{definition}
		Let $\mathbb{X}$ be a Banach space. We say that  $\mathbb{X}$ has the \textit{Property P} if there exists an $\epsilon\in (0,1)$ such that   the following two conditions are equivalent:\\
		(i)  $T\in \mathbb{L}(\mathbb{X})$ preserves $\epsilon$-orthogonality  at each $x\in \mathbb{X}.$\\
		(ii) $T\in \mathbb{L}(\mathbb{X})$ is a scalar multiple of an isometry.
	\end{definition}
	
	Clearly, whenever a Banach space $ \mathbb{X} $ possesses the \textit{Property P}, a refinement of the Blanco-Koldobsky-Turn\v{s}ek theorem is possible for $ \mathbb{X}. $ On the other hand, if $ \mathbb{X} $ does not possess the \textit{Property P}, it follows that the Blanco-Koldobsky-Turn\v{s}ek theorem can not be improved by considering approximate preservation of orthogonality. Our main aim is to produce concrete examples of both types of Banach spaces by considering certain geometric properties of the concerned space. We first  investigate some  geometric and analytic properties of finite-dimensional polyhedral Banach spaces, including the smoothness and the number of extreme supporting functionals of each points, related to such preservation. Additionally,  we analyze the linear operators satisfying such preservation of orthogonality, in connection with the geometry of the concerned spaces.  As an application of this study, we characterize the isometries on some Banach spaces, thus obtaining refinements of the Blanco-Koldobsky-Turn\v{s}ek theorem on these spaces. Finally, we establish that  a two-dimensional Banach space, whose unit sphere is given by a regular $2n$-gon $(n>2),$ and the  sequence space $\ell_{\infty}^n~(n>2)$ do satisfy the \textit{Property P}. On the other hand, we prove that the sequence spaces $\ell_p~(1\leq p< \infty)$ do not possess this property.
    
    We end this section with the following characterization of the approximate Birkhoff-James orthogonality, which will be used repeatedly.
	\begin{theorem}\cite[Th. 2.4]{CSW17}\label{charac}
		Let $\mathbb {X}$  be a Banach space and let $x,y\in \mathbb{X}.$ Let $\epsilon\in[0,1).$ Then $x$ is $\epsilon$-approximate Birkhoff-James orthogonal to $y$ if and only if there exists $f\in J(x)$ such that $|f(y)|\leq \epsilon \|y\|.$
	\end{theorem}
	\section{Main Results}


	 We begin this section with the following proposition.  
	
	\begin{proposition}\label{propo}
		Let $\mathbb{X}$ be a finite-dimensional polyhedral Banach space. For each $f \in Ext~B_{\mathbb{X}^*}$, consider the set $\mathcal{A}_f=\{x \in \mathbb{X}: J(x)=\{f\}\}$. Then the following results hold true:
		\begin{itemize}
			\item[(i)] For each $f \in$ Ext $B_{\mathbb{X}^*}$, $\mathcal{A}_f$ is nonempty, open and a convex cone.
			\item[(ii)] For each  non-zero $x\in \mathbb{X},$ $f\in Ext~J(x)$ if and only if $x\in \overline{\mathcal{A}}_f .$
			\item[(iii)] $\mathbb{X}=\bigcup\limits_{f \in Ext~B_{\mathbb{X}^*}}\overline{\mathcal{A}}_f$.
			\item[(iv)] If $f=-g$, then $\overline{\mathcal{A}}_f \cap \overline{\mathcal{A}}_g=\{0\}$.
			\item[(v)] Let $F$ be the facet of $B_{\mathbb{X}}$ corresponding to $f \in Ext~ B_{\mathbb{X}^*}$, then $\mathcal{ A}_f=\{r x: r>0\text{ and }x \in \text{Int}_r~ F\},$ i.e.,  $\mathcal{A}_f\cup \{0\}$  is the conical hull of $\text{Int}_r~ F.$ 
			
		\end{itemize}
	\end{proposition}
	\begin{proof}
		(i) First we show that for each  $f \in Ext~B_{\mathbb{X}^*},$ $\mathcal {A}_f$ is nonempty. Let $f \in Ext~B_{\mathbb{X}^*}.$  Since $\mathbb{X}$ is finite-dimensional polyhedral Banach space, it follows that  there exists $x(\neq 0)\in \mathbb{X}$ such that $g(x)< f(x)=\|x\|$ for all $g\in S_{\mathbb{X}^*}.$ Thus, $J(x)=\{f\}$ and so $x\in \mathcal{A}_f.$ Therefore, $\mathcal{A}_f$ is nonempty.
		
		Next, we show that for each  $f \in Ext~B_{\mathbb{X}^*},$ $\mathcal {A}_f$ is open. Let $f \in Ext~ B_{\mathbb{X}^*}$ and
		$\mathcal{A}_f=\{x \in \mathbb{X}: J(x)=\{f\}\}.$ We observe that if $y\in \mathcal{A}_f^c$ then either $y=0$ or there exists $g(\neq f)\in Ext~B_{\mathbb{X}^*} $ such that $g(y)=\|y\|.$ Hence, we have
		\[	\mathcal{A}_f^c=\bigcup\limits_{g(\neq\pm f) \in Ext~B_{\mathbb{X}^*}}\{y \in \mathbb{X}: g(y)=\|y\|\} .\]
		Now, for each $g \in Ext~ B_{\mathbb{X}^*},$ the set $\{y \in \mathbb{X}: g(y)-\|y\|=0\}$ is closed. Since $\mathbb{X}$ is finite-dimensional, $Ext~B_{\mathbb{X}^*}$ is finite and hence $\mathcal{A}_f^c$ being a finite union of closed sets is closed. Therefore, $\mathcal{A}_f$ is open.
		
		Next, we show that for each $f \in Ext~ B_{\mathbb{X}^*},$ $ \mathcal{A}_f$ is a  convex cone.	Let $f \in Ext~ B_{\mathbb{X}^*}$ and let $x, y \in \mathcal{A}_f.$ Then $f(x)=\|x\|\text{ and }f(y)=\|y\|.$ Clearly, $\alpha x+\beta y \neq 0$ for all $\alpha,\beta> 0.$ Now,	$	\alpha\|x\|+\beta\|y\|  =f(\alpha x+\beta y) \leq\|\alpha x+\beta y\|\leq\alpha \|x\|+\beta \|y\| $ and so $f(\alpha x+\beta y)=\|\alpha x+\beta y\|=\alpha\|x\|+\beta\|y\|.$ Hence $f \in J(\alpha x+\beta y).$ We claim that $J(\alpha x+\beta y)=\{f\}.$ Let  $g \in J(\alpha x+\beta y).$  Observe that $g(\alpha x)\leq \alpha \| x\|$ and $g(\beta y)\leq\beta\| y\|$ and $ g(\alpha x)+g(\beta y)=g(\alpha x+\beta y) =\|\alpha x+\beta y\| =\alpha \|x\|+\beta\|y\|.$ This implies that $g(x)=\|x\|$ and $g(y)=\|y\|$ and so $g=f.$ Hence $J(\alpha x+\beta y)=\{f\}.$ Therefore, $\alpha x+\beta y \in \mathcal{A}_f$ for all $\alpha,\beta>0.$ Thus, $\mathcal{A}_f$ is a convex cone.\\
		
		(ii)   Let $x(\neq 0)\in \mathbb{X}$ and let $f \in Ext~B_{\mathbb{X}^*}$ be such that $x\in \overline{ \mathcal{A}}_f.$ Then there exists a sequence $\{z_n\} \subset \mathcal{A}_f$ such that $z_n \longrightarrow z$ as $n\longrightarrow\infty.$
		So $\|z_n\| \longrightarrow\|z\|$ as $n\longrightarrow \infty.$
		Now, $f(z_n) \longrightarrow f(z)$ implies that $f(z)=\|z\|.$ Thus, $f\in Ext~J(x).$
		
		Conversely, let $x(\neq 0)\in \mathbb{X}$ and let $f \in Ext~J(x).$ We show that $x \in \overline{\mathcal{A}}_f$.	Let $y \in \mathcal{A}_f$. Consider the sequence $\Big\{\Big(1-\frac{1}{n}\Big) x+\frac{1}{n} y\Big\}$. We claim that $\Big(1-\frac{1}{n}\Big) x+\frac{1}{n} y\in\mathcal{A}_f$ for all $n\in \mathbb N.$
		On the contrary to our claim, suppose that $\Big(1-\frac{1}{n_0}\Big) x+\frac{1}{n_0} y \notin \mathcal{A}_f$  for some $n_0 \in \mathbb{N}.$ Then there exists $g \in J\Big(\Big(1-\frac{1}{n_0}\Big) x+\frac{1}{n_0} y\Big)$ for some $g(\neq f) \in Ext~ B_{\mathbb{X}^*}.$ Now, $f(x)=\|x\|$ and $f(y)=\|y\|$  imply that 
		\begin{eqnarray*}
			\Big(1-\frac{1}{n_0}\Big)\|x\|+\frac{1}{n_0}\|y\| & =&\Big|f\Big(\Big(1-\frac{1}{n_0}\Big) x+\frac{1}{n_0} y\Big)\Big| \\
			& \leq&\Big\|\Big(1-\frac{1}{n_0}\Big) x+\frac{1}{n_0} y\Big\| \\
			&\leq& \Big(1-\frac{1}{n_0}\Big)\|x\|+\frac{1}{n_0}\|y\|.
		\end{eqnarray*}
		Hence, \[\Big(1-\frac{1}{n_0}\Big)\|x\|+\frac{1}{n_0}\|y\|=\Big\|\Big(1-\frac{1}{n_0}\Big) x+\frac{1}{n_0} y\Big\|.\]
		We next observe that
		\begin{eqnarray*}
			&&g\Big(\Big(1-\frac{1}{n_0}\Big) x+\frac{1}{n_0} y\Big)=\Big\|\Big(1-\frac{1}{n_0}\Big) x+\frac{1}{n_0} y\Big\| \\
			&\implies & \Big(1-\frac{1}{n_0}\Big) g(x)+\frac{1}{n_0} g(y)=\Big(1-\frac{1}{n_0}\Big)\|x\|+\frac{1}{n_0}\|y\| . \\
			&\implies & g(y)=\|y\| .
		\end{eqnarray*}
		This contradicts the fact that $y \in \mathcal{A}_f.$
		Thus, the sequence $\Big\{\Big(1-\frac{1}{n}\Big) x+\frac{1}{n} y\Big\} \subset \mathcal{A}_f$ and so $x \in \overline{\mathcal{A}}_f.$\\
		(iii) Clearly, $\bigcup\limits_{f \in Ext~B_{\mathbb{X}^*}}\overline{\mathcal{A}}_f\subset\mathbb{X}.$ For the converse part, if $x(\neq 0)\in \mathbb{X}$ then there exists $f\in Ext~B_{\mathbb{X}^*}$ such that $f(x)=\|x\|.$ Then $f\in Ext~J(x)$ and so from (ii), it follows that $x\in \overline{\mathcal{A}}_f.$ Again, it is easy to observe that  $0\in \overline{\mathcal{A}}_f.$ Thus,  $\mathbb{X}\subset\bigcup\limits_{f \in Ext~B_{\mathbb{X}^*}}\overline{\mathcal{A}}_f.$ This completes the proof.\\
		(iv) Let $f=-g.$
		We show that $\overline{\mathcal{A}}_f \cap \overline{\mathcal{A}}_g=\{0\}.$
		Suppose on the contrary that there exists a non-zero $z \in \overline{\mathcal{A}}_f \cap \overline{\mathcal{A}}_g.$ From (ii), we have  $f,g\in Ext~J(z),$
		which is a contradiction. Therefore, $\overline{\mathcal{A}}_f \cap \overline{\mathcal{A}}_g=\{0\}.$\\
		(v) Let $f \in Ext~ B_{\mathbb{X}^*}$ and let $F$ be the facet of $B_{\mathbb{X}}$ corresponding to $f.$ We show that $\mathcal{A}_f=\{r x: r>0\text{ and }x \in \text{Int}_r~ F\}.$ Let $y \in \text{Int}_r~ F.$ Then $y$ is a smooth point of $\mathbb{X}$ and so $J(y)=\{f\}.$
		Thus, $\{r x: r>0\text{ and }x \in \text{Int}_r~ F\} \subset \mathcal{A}_f.$
		Let $z \in \mathcal{A}_f,$ then $z \neq 0$ and
		$ J(z)=\{f\} .$
		So $\frac{z}{\|z\|} \in  \text{Int}_r~ F.$ This implies that $z \in \{r x: r>0\text{ and }x \in \text{Int}_r~ F\}.$
		Thus, $\mathcal{A}_f \subset\{r x: r>0\text{ and }x \in \text{Int}_r~ F\}$ and this completes the proof.
	\end{proof}

	In the next two results, we observe that if $\mathbb{X}=\ell_{\infty}^n $ or $\ell_{1}^n, $ then for any $0\leq\epsilon<1,$ the intersection of the $\epsilon$-orthogonality sets of two smooth points of $\mathbb{X}$  does not contain any hyperspace, unless the supporting functionals of these smooth points are linearly dependent.
	\begin{theorem}\label{linfhyper}
		Let $\mathbb{X}=\ell_{\infty}^n.$ Let $x, y \in \mathbb{X}$ be smooth points of $\mathbb{X}$ such that $J(x) \neq \pm J(y)$. Then for any  $0\leq\epsilon<1,$ the set $x^{\perp_B^\epsilon} \cap y^{\perp_B^\epsilon}$ does not contain any hyperspace.	
	\end{theorem} 
	\begin{proof}
		If $dim(\mathbb{X})=2,$ it can be easily seen that $x^{\perp_B^\epsilon} \cap y^{\perp_B^\epsilon}=\{0\}$ and therefore, $x^{\perp_B^{\epsilon}} \cap y^{\perp_B^{\epsilon}}$ does not contain any hyperspace. Let us assume that $dim(\mathbb{X})>2.$  Note that $Ext~B_{\mathbb{X}^*}=\{\pm f_1,\pm f_2,\dots,\pm f_n\},$ where $f_i(\alpha_1,\alpha_2,\dots,\alpha_n)=\alpha_i$ for $1\leq i\leq n.$ Without loss of generality, we assume that $J(x)=\{f_p\}$ and $J(y)=\{f_q\}$ for distinct $1\leq p,q\leq n.$ Observe that if a non-zero $(z_1,z_2,\dots,z_n)\in x^{\perp_B^\epsilon} \cap y^{\perp_B^\epsilon},$ then using Theorem \ref{charac}, we get $|f_p(z_1,z_2,\dots,z_n)|\leq \epsilon\max\{|z_i|:1\leq i\leq n\} \text{ and }|f_q(z_1,z_2,\dots,z_n)|\leq \epsilon\max\{|z_i|:1\leq i\leq n\}$ and so $|z_p|,|z_q|<\{|z_i|:i\in\{1,2,\dots,n\}\setminus\{p,q\}\}.$ 
		Consider the two-dimensional subspace
		\[	\mathbb{V}=\{(v_1,v_2,\dots,v_n):v_i=0~\forall~i\in\{1,2,\dots,n\}\setminus\{p,q\}\}.\]
		Clearly, $\mathbb{V} \cap(x^{\perp_B^{\epsilon}} \cap y^{\perp_B^{\epsilon}})=\{0\}$.
		Since $\mathbb{V}$ is a two-dimensional subspace of $\mathbb{X}$ and $\mathbb{V} \cap(x^{\perp_B^{\epsilon}} \cap y^{\perp_B^{\epsilon}})=\{0\}$, it follows that $x^{\perp_B^{\epsilon}} \cap y^{\perp_B^{\epsilon}}$ does not contain any hyperspace.
	\end{proof}
	\begin{theorem}
		Let $\mathbb{X}=\ell_{1}^n.$ Let $x, y \in \mathbb{X}$ be two smooth points of $\mathbb{X}$ such that $J(x) \neq \pm J(y)$. Then for any  $0\leq\epsilon<1,$ the set $x^{\perp_B^\epsilon} \cap y^{\perp_B^\epsilon}$ does not contain any hyperspace.	
	\end{theorem} 
	\begin{proof}
		Let $x=(x_1,x_2,\dots,x_n), y=(y_1,y_2,\dots,y_n) \in \mathbb{X}$ be two smooth points of $\mathbb{X}$ such that $J(x) \neq \pm J(y)$.  If $dim(\mathbb{X})=2,$ it can be easily seen as before that $x^{\perp_B^\epsilon} \cap y^{\perp_B^\epsilon}=\{0\}$ and therefore, $x^{\perp_B^{\epsilon}} \cap y^{\perp_B^{\epsilon}}$ does not contain any hyperspace. Now, let $J(x)=\{f\}$ and $J(y)=\{g\}$, where for each $z=(z_1,z_2,\dots,z_n)$,
		\[f(z)=\sum\limits_{i=1}^n sgn(x_i) z_i~\text{ and }~ g(z)=\sum\limits_{i=1}^n sgn(x_i) z_i.\]
		Since $f$ and $g$ are linearly independent, there exist $1 \leq p, q \leq n$ such that
		\[sgn(x_p)=sgn(y_p)\text{ and } sgn(x_q)=-sgn(y_q).\]
		Consider the two-dimensional subspace
		\[	\mathbb{V}=\{(v_1,v_2,\dots,v_n):v_i=0~\forall~i\in\{1,2,\dots,n\}\setminus\{p,q\}\}.\] 
		We claim that $\mathbb{V} \cap(x^{\perp_B^{\epsilon}} \cap y^{\perp_B^{\epsilon}})=\{0\}$.
		On the contrary to our claim, suppose that there exists a non-zero  $v=(v_1,v_2,\dots,v_n)\in \mathbb{V}\cap x^{\perp_B^\epsilon} \cap y^{\perp_B^\epsilon} .$ From Theorem \ref{charac}, it follows that $|f(v)| \leq \epsilon\|v\|$ and
		$|g(v)| \leq \epsilon\|v\|.$
		Thus,
		\[	|f(v)|=|sgn(x_p)v_p+sgn(x_q)v_q|\text{ and }	|g(v)|=|sgn(y_p)v_p+sgn(y_q)v_q|.\]
		
		This shows that either $|f(v)|=|v_p|+|v_q|=\|v\|$ or $|g(v)|=|v_p|+|v_q|=\|v\|.$ As $\epsilon<1,$ this contradicts the fact that$|f(v)| \leq \epsilon\|v\|$ and
		$|g(v)| \leq \epsilon\|v\|.$ Thus our claim is established.
		Since $\mathbb{V}$ is a two-dimensional subspace of $\mathbb{X}$ and $\mathbb{V} \cap(x^{\perp_B^{\epsilon}} \cap y^{\perp_B^{\epsilon}})=\{0\}$, it follows that $(x^{\perp_B^{\epsilon}} \cap y^{\perp_B^{\epsilon}})$ does not contain any hyperspace.
	\end{proof}
	The previous two results may not hold true in a general Banach space, which is illustrated in the following example. 
	\begin{example}
		Let  $\mathbb{X}$ be the three-dimensional polyhedral Banach space whose unit sphere is the  octagonal prism with vertices $\{u_i,v_i:1\leq i\leq8\},$ where
		$$u_i=\Big(\cos(i-1)\frac{\pi}{4},\sin(i-1)\frac{\pi}{4},1\Big)\text{ and }v_i=\Big(\cos(i-1)\frac{\pi}{4},\sin(i-1)\frac{\pi}{4},-1\Big).
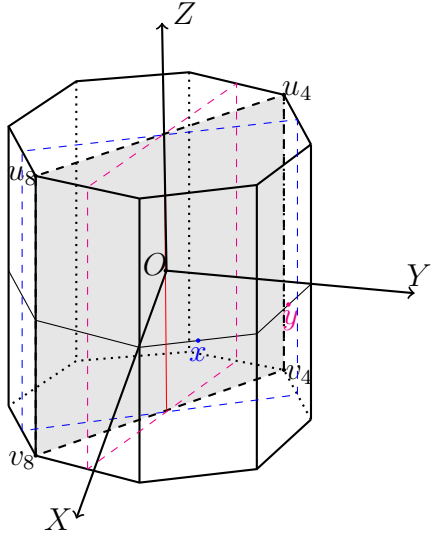
$$
		\begin{center}
			\begin{tikzpicture}[scale=.6]
				\filldraw[fill=black!10,dashed,thick] (-2.9,-4.1)--(-2.9,2.1)--(2.6,3.9)--(2.6,-2.2)--(-2.9,-4.1);
				\draw[black,thick](3.2,2.8)--(2.6,3.9)--(.5,4.4)--(-2,4.2)--(-3.5,3.2)--(-2.9,2.1)--(-.6,1.6)--(2,1.9)--(3.2,2.8);
				\draw[black,dotted,thick](3.2,-3.3)--(2.6,-2.2)--(.5,-1.8)--(-2,-2)--(-3.5,-3);
				\draw[black,thick](-3.5,-3)--(-2.9,-4.1)--(-.6,-4.7)--(2,-4.4)--(3.2,-3.3);
				\draw[black](-3.5,0)--(-2.9,-1.1)--(-.6,-1.7)--(2,-1.4)--(3.2,-.3);
				\draw[red](-0.05,3)--(-0,-3.1);
				\draw[black,thick](3.2,2.8)--(3.2,-3.3);
				\draw[black,thick](2,1.9)--(2,-4.4);
				\draw[black,thick](-.6,1.6)--(-.6,-4.7);
				\draw[black,thick](-2.9,2.1)--(-2.9,-4.1);
				\draw[black,thick](-3.5,3.2)--(-3.5,-3);
				\draw[black,dotted,thick](2.6,3.9) -- (2.6,-2.2);
				\draw[black,dotted,thick](.5,4.4) -- (.5,-1.8);
				\draw[black,dotted,thick](-2,4.2) -- (-2,-2);
				\fill[fill=blue] (.7,-1.55) circle (1.4 pt);
				\node [blue] at (.7,-1.85) {$x$};	
				\fill[fill=magenta] (2.7,-.75) circle (1.4pt);
				\node [magenta] at (2.75,-1.1) {$y$};	
				\fill[fill=black] (-0.02,0) circle (1.3 pt);
				\node [black] at (-.25,.15) {$O$};	
				\fill[fill=black] (-2.9,-4.1) circle (1.3pt);
				\node [black] at (-3.2,-4.2) {$v_8$};	
				\fill[fill=black] (2.6,3.9) circle (1.3pt);
				\node [black] at (2.9,4) {$u_4$};	
				\fill[fill=black] (-2.9,2.1) circle (1.3pt);
				\node [black] at (-3.2,2.1) {$u_8$};	
				\fill[fill=black] (2.6,-2.2) circle (1.3pt);
				\node [black] at (2.95,-2.3) {$v_4$};	
				\draw[blue,dashed](2.9,-2.75)--(2.9,3.35) -- (-3.2,2.65)--(-3.2,-3.55)--(2.9,-2.75);
				\draw[dashed,magenta](-1.75,-4.4)--(-1.75,1.85) -- (1.55,4.15)--(1.55,-2)--(-1.75,-4.4);
				\draw[->, black,thick ](0,0)--(-.1,5.5);
				\draw[->, black,thick ](0,0)--(5.5,-0.5);
				\draw[->, black,thick ](0,0)--(-2,-5.5);
				\node [black] at (.4, 5.7) {$Z$};	
				\node [black] at (5.6, -.1) {$Y$};	
				\node [black] at (-2.4, -5.5) {$X$};	
			\end{tikzpicture}
		\end{center}
		Let $x=\big(\frac12+\frac1{2\sqrt{2}},\frac1{2\sqrt{2}},0\big),~y=\big(\frac1{2\sqrt{2}},\frac12+\frac1{2\sqrt{2}},0\big).$ Then $J(x)=\{f\}\text{ and }J(y)=\{g\},$ where $f(\alpha,\beta,\gamma)=\frac{1}{\cos \frac{\pi}{8}}\left(\alpha\cos \frac{\pi}{8}+\beta\sin \frac{\pi}{8}\right)$ and  $g(\alpha,\beta,\gamma)=\frac{1}{\cos \frac{\pi}{8}}\left(\alpha\cos \frac{3\pi}{8}+\beta\sin \frac{3\pi}{8}\right).$ Let $\mathbb{V}=span~\{u_8,v_8\}.$  We show that  the hyperspace $\mathbb{V}\subset x^{\perp_B^\epsilon} \cap y^{\perp_B^\epsilon}$ for $\sqrt2-1\leq\epsilon<1.$ Since Birkhoff-James orthogonality is homogeneous, it is sufficient to show that $\mathbb{V}\cap S_{\mathbb{X}}\subset x^{\perp_B^\epsilon} \cap y^{\perp_B^\epsilon}.$ Clearly, $\mathbb{V}\cap S_{\mathbb{X}}=co(\{u_8,v_8\})\cup co(\{v_8,v_4\})\cup co(\{v_4,u_4\})\cup co(\{u_4,u_8\}).$ As $u_4=-v_8$ and $u_8=-v_4,$ we have $co(\{u_8,v_8\})=-co(\{v_4,u_4\})$ and $co(\{v_8,v_4\})=-co(\{u_4,u_8\}).$ Therefore, it is enough to show that $co(\{u_8,v_8\})\cup co(\{v_8,v_4\})\subset x^{\perp_B^\epsilon} \cap y^{\perp_B^\epsilon}.$\\
		First, we show that $co(\{u_8,v_8\})\subset x^{\perp_B^\epsilon} \cap y^{\perp_B^\epsilon}.$ Let $u\in co(\{u_8,v_8\}).$ Then $u=(1-t)u_8+tv_8=\big(\frac1{\sqrt{2}},-\frac1{\sqrt{2}},1-2t\big),$ for some $t\in [0,1].$ Then
		\[|f(u)|=\Big|\frac{1}{\cos \frac{\pi}{8}}\Big(\frac1{\sqrt{2}}\cos \frac{\pi}{8}-\frac1{\sqrt{2}}\sin \frac{\pi}{8}\Big)\Big|=\sqrt2-1\leq\epsilon=\epsilon\|u\|.\]
		Again,
		\[|g(u)|=\Big|\frac{1}{\cos \frac{\pi}{8}}\Big(\frac1{\sqrt{2}}\cos \frac{3\pi}{8}-\frac1{\sqrt{2}}\sin \frac{3\pi}{8}\Big)\Big|=\sqrt2-1\leq\epsilon=\epsilon\|u\|.\]
		Thus, from Theorem \ref{charac}, it follows that $u\in x^{\perp_B^\epsilon} \cap y^{\perp_B^\epsilon}.$ Therefore,  $co(\{u_8,v_8\})\subset x^{\perp_B^\epsilon} \cap y^{\perp_B^\epsilon}.$
		
		Next, we show that $co(\{v_8,v_4\})\subset x^{\perp_B^\epsilon} \cap y^{\perp_B^\epsilon}.$ Let $v\in co(\{v_8,v_4\}).$ Then $v=(1-t)v_8+tv_4=\big(\frac{1}{\sqrt{2}}(1-2t),-\frac{1}{\sqrt{2}}(1-2t),-1\big),$ for some $t\in [0,1].$ Then
		$$|f(v)|=\Big|\frac{1}{\cos \frac{\pi}{8}}\Big(\frac{1}{\sqrt{2}}(1-2t)\cos \frac{\pi}{8}-\frac{1}{\sqrt{2}}(1-2t)\sin \frac{\pi}{8}\Big)\Big|\leq\sqrt2-1\leq\epsilon=\epsilon\|v\|.$$ Again,
		$$|g(v)|=\Big|\frac{1}{\cos \frac{\pi}{8}}\Big(\frac{1}{\sqrt{2}}(1-2t)\cos \frac{3\pi}{8}-\frac{1}{\sqrt{2}}(1-2t)\sin \frac{3\pi}{8}\Big)\Big|\leq\sqrt2-1\leq\epsilon=\epsilon\|v\|.$$
		Thus, from Theorem \ref{charac}, it follows that $u\in x^{\perp_B^\epsilon} \cap y^{\perp_B^\epsilon}.$ Therefore,  $co(\{v_8,v_4\})\subset x^{\perp_B^\epsilon} \cap y^{\perp_B^\epsilon}.$\\
	\end{example}

	Our next result gives the existence of an $\epsilon>0,$ for which the intersection of the $\epsilon$-orthogonality sets of two smooth points  of a normed linear space does not contain any hyperspace, provided the supporting functionals of these smooth points are linearly independent.
	\begin{theorem}\label{hyperspace}
		Let $\mathbb{X}$ be a normed linear space. Let $x,y\in \mathbb{X}$  be two smooth points  such that $J(x)=\{f\}$ and $J(y)=\{g\},$ where $f,g$ are linearly independent. Suppose that 
		\[K_1=\sup\{|f(z)|:z\in\ker g\cap S_{\mathbb{X}}\}\text{ and }K_2=\sup\{|g(w)|:w\in\ker f\cap S_{\mathbb{X}}\}.\]
		Then for any $0\leq \epsilon<\frac12 \min\{K_1,K_2\},$ there does not exist any hyperspace contained in $x^{\perp_B^\epsilon} \cap y^{\perp_B^\epsilon}.$
	\end{theorem}
	\begin{proof}
		Since $f$ and $g$ are linearly independent, it follows that there exist $z\in \ker g\setminus \ker f$ and $w\in \ker f\setminus \ker g$ and so $K_1,K_2>0.$ Let
		\[K=\frac12 \min\{K_1,K_2\}.\]
		For $\epsilon=0,$ the theorem holds trivially. Now we show that for each $0<\epsilon<K,$ $x^{\perp_B^\epsilon} \cap y^{\perp_B^\epsilon}$ does not contain any hyperspace. On the contrary suppose that there exists $0<\epsilon_0<K$ such that $x^{\perp_B^{\epsilon_0}} \cap y^{\perp_B^{\epsilon_0}}$ contains a hyperspace. Let $\mu=K-\epsilon_0.$ Then $0<\epsilon_0<K-\frac{\mu}{2}.$ Since $K_1=\sup\{|f(z)|:z\in\ker g\cap S_{\mathbb{X}}\}\text{ and }K_2=\sup\{|g(w)|:w\in\ker f\cap S_{\mathbb{X}}\},$ there exist $u\in \ker g\cap S_{\mathbb{X}} $ and $v\in \ker f\cap S_{\mathbb{X}} $ such that 
		\[|f(u)|\geq K_1-\mu\text{ and }|g(v)|\geq K_2-\mu.\]
		Consider $\mathbb{V}=span~\{u,v\}.$ Since $x^{\perp_B^{\epsilon_0}} \cap y^{\perp_B^{\epsilon_0}}$ contains a hyperspace, it follows that $(x^{\perp_B^{\epsilon_0}} \cap y^{\perp_B^{\epsilon_0}})\cap \mathbb{V}\neq\{0\}$ and so there exist $\alpha,\beta\in \mathbb{R}$ (not both zero) such that $\alpha u+\beta v \in  x^{\perp_B^{\epsilon_0}} \cap y^{\perp_B^{\epsilon_0}}.$ From Theorem \ref{charac}, it follows that 
		\begin{eqnarray*}
			&\text{(i)}&	|\alpha|(2K-\mu)\leq|\alpha|(K_1-\mu)\leq |f(\alpha u+\beta v )|\leq \epsilon_0\|\alpha u+\beta v \|\leq\epsilon_0(|\alpha|+|\beta|),\\
			&\text{(ii)}&|\beta|(2K-\mu)\leq|\beta|(K_2-\mu)\leq|g(\alpha u+\beta v )|\leq \epsilon_0\|\alpha u+\beta v \|\leq\epsilon_0(|\alpha|+|\beta|).
		\end{eqnarray*}
		Hence, $(2K-\mu)
		\leq2\epsilon_0<(2K-\mu),$ which is not possible. Therefore, $x^{\perp_B^{\epsilon_0}} \cap y^{\perp_B^{\epsilon_0}}$ does not contain any hyperspace and this completes the proof.
	\end{proof}
	Using the above result and \cite[Th. 3.2]{S21}, we have the following corollary. 
	\begin{cor}
		Let $\mathbb{X}$ be a reflexive Banach space. Let $x,y\in \mathbb{X}$  be smooth points  such that $J(x)=\{f\}$ and $J(y)=\{g\},$ where $f,g$ are mutually orthogonal in the sense of Birkhoff-James.
		Then for any $0\leq\epsilon<\frac12,$ there does not exist any hyperspace contained in $x^{\perp_B^\epsilon} \cap y^{\perp_B^\epsilon}.$
	\end{cor}

	Let $\mathbb{X} $ be an $n$-dimensional polyhedral Banach space. For each $f \in Ext~ B_{\mathbb{X}^*},$ consider the set $\mathcal{A}_f=\{x \in \mathbb{X}: J(x)=\{f\}\}.$ From Theorem \ref{hyperspace}, it follows that if $f,g\in Ext~ B_{\mathbb{X}^*}$ are linearly independent, then there exists $\epsilon_{fg}>0$ such that $x^{\perp_B^{\epsilon_{fg}}}\cap y^{\perp_B^{\epsilon_{fg}}}$ does not contain any hyperspace, for any $x\in \mathcal{A}_f$ and $y\in \mathcal A_g.$ We define $\epsilon_{\mathbb{X}}$ by
	\[\epsilon_{\mathbb{X}}= \min\{\epsilon_{fg}: f,g\in  Ext~ B_{\mathbb{X}^*} \text{ are linearly independent}\}.\]
	Clearly, $\epsilon_\mathbb{X}\leq 1.$    In particular, for $\mathbb{X}=\ell_{\infty}^n(\text{or,}~\ell_{1}^n),$ we take $\epsilon_{\mathbb{X}}=1.$
	Then for any linearly independent $f,g\in Ext~ B_{\mathbb{X}^*},$  $x^{\perp_B^{\epsilon}}\cap y^{\perp_B^{\epsilon}}$ does not contain any hyperspace, for any $x\in \mathcal{A}_f$ and $y\in \mathcal A_g.$\\

	In the following lemma, we show that for preservation of $\epsilon$-orthogonality  by a non-zero operator,  it is necessary for the operator to be injective. 
	\begin{lemma}\label{bijective}
		Let $\mathbb{X},\mathbb{Y}$ be normed linear spaces. If  a non-zero operator $T \in \mathbb{L}(\mathbb{X},\mathbb{Y})$ preserves $\epsilon$-orthogonality at each $x \in\mathbb{X}$ for any $0 \leq \epsilon<1,$ then $T$ is one-to-one.
	\end{lemma}
	\begin{proof}
		Let $0 \leq \epsilon<1$ and let $T(\neq 0)\in\mathbb{L}(\mathbb{X},\mathbb{Y})$ preserve $\epsilon$-orthogonality at each $x \in \mathbb{X}$. If possible suppose that $T$ is not one-to-one.
		Let $u \in \ker T$ and $v\in \mathbb{X}\setminus \ker T$ be such that $\|u\|=\|v\|=1.$ Suppose that $\mathbb{V}=span\{u,v\}.$ 
		Choose  $y \in S_{\mathbb{V}}$ such that $\{u,y\}$ is linearly independent and $\|y-u\|<1.$ Let $z(\neq 0)\in \mathbb{V}$ be such that $y\perp_B z.$ This implies that $z \text{ and } u$ are linearly independent. It is easy to observe that $y,z\notin \ker T.$
		Now, $y \perp_B z \implies T y \perp_B^\epsilon T z.$
		Since $T|_V:\mathbb{V}\longrightarrow T(\mathbb{V})$ is of rank one, $Ty=\lambda Tz$ for some $\lambda\in \mathbb{R}\setminus\{0\}$. This  contradicts the fact that $T y \perp_B^\epsilon T z.$ Therefore, $T$ is one-to-one.
	\end{proof}
	
	We also require the following lemma, which will be used to establish that the linear image corresponding to a facet of the unit ball in a finite-dimensional polyhedral Banach space fully lies in the conical hull of a facet under the preservation of $\epsilon$-orthogonality.
	
	\begin{lemma}\label{openmap}
		Let $\mathbb{X}, \mathbb{Y}$ be Banach spaces and let $T \in \mathbb{L}(\mathbb{X},\mathbb{Y})$ be bijective. Suppose that $A$ and $B$ are two open sets of $\mathbb{X}$ and $\mathbb{Y},$ respectively. If $T(A) \cap \overline{B} \neq \emptyset$ then $T(A) \cap B \neq \emptyset.$
	\end{lemma}
	\begin{proof}
		Suppose that $T \in \mathbb{L}(\mathbb{X},\mathbb{Y})$ is such that $T(A) \cap \overline{B} \neq \emptyset.$
		Let $x \in A$ be such that $T x \in \overline{B}.$
		If $T x \in B$ then we are done. Let $T x \notin B.$ Then $T x$ is a limit point of $B$. Since $T$ is continuous and bijective, by  the open mapping theorem $T(A)$ is open. Now, $T(A)$ is an open set containing the limit point $T x$ of $B.$
		Thus, $T(A) \cap B \neq \emptyset.$
	\end{proof}
	\begin{remark}
		If $A$ is not open then the above lemma may not hold true. For example, let $\mathbb{X}$ be a Banach space and let $A=S_{\mathbb{X}},B= B_{\mathbb{X}}\setminus S_{\mathbb{X}}.$ Consider the identity operator $T\in \mathbb{L}(X).$ Then $T(A) \cap \overline{B} \neq \emptyset,$ but $T(A) \cap B = \emptyset.$  
	\end{remark}
	
	We are now ready to prove the desired result.
	\begin{theorem}\label{facetTofacet}
		Let $\mathbb{X},\mathbb{Y}$ be $n$-dimensional polyhedral Banach  spaces. For each $f \in Ext~ B_{\mathbb{X}^*}$, consider the set $\mathcal{A}_f=\{x \in \mathbb{X}: J(x)=\{f\}\}$ and  for each $g \in Ext~ B_{\mathbb{Y}^*}$, consider the set $\mathcal{B}_g=\{y \in \mathbb{Y}: J(y)=\{g\}\}.$ Let $0\leq \epsilon<\epsilon_{\mathbb{Y}}.$  If a non-zero operator $T\in \mathbb{L}(\mathbb{X},\mathbb{Y})$ preserves $\epsilon$-orthogonality at each $x \in \mathbb{X}$, then for each $f \in  Ext~ B_{\mathbb{X}^*}, T(\mathcal{A}_f) \subset \mathcal{B}_g$ for some $g \in  Ext~ B_{\mathbb{Y}^*}.$ Moreover, this $g$ determined by $f$ is unique.  
	\end{theorem}
	\begin{proof}
		Let $T(\neq 0) \in \mathbb{L}(\mathbb{X},\mathbb{Y})$ preserve $\epsilon$-orthogonality at each $x \in \mathbb{X}$.
		It follows from Lemma \ref{bijective} that $T$ is bijective.
		From Proposition \ref{propo}, it follows that
		$\bigcup\limits_{g \in Ext~ B_{\mathbb{Y}^*}}\overline{ \mathcal{B}}_g=\mathbb{Y}$ and $\overline{ \mathcal{B}}_g \cap \overline{ \mathcal{B}}_{-g}=\{0\}$.
		Since for each $f \in Ext~ B_{\mathbb{X}^*},~ \mathcal{A}_f$ is connected, it follows that $T(\mathcal{A}_f)$ is also connected. We observe that if $T(\mathcal{A}_f) \subset \overline{\mathcal{B}}_g \cup \overline{ \mathcal{B}}_{-g},$ then either  $T(\mathcal{A}_f) \subset \overline{\mathcal{B}}_g $ or $T(\mathcal{A}_f) \subset \overline{ \mathcal{B}}_{-g}.$
		We claim that for each $f \in Ext~ B_{\mathbb{X}^*}, T(\mathcal{A}_f)\subset\overline{\mathcal{B}}_g$ for some $g \in Ext~ B_{\mathbb{Y}^*}.$
		On the contrary to our claim, suppose that there exist $x_1, x_2 \in \mathcal{A}_f$ such that
		\[	T x_1 \in \overline{ \mathcal{B}}_{g_1} \text { and } T x_2 \in \overline {\mathcal{B}}_{g_2}\text{ for some } g_1, g_2(\neq\pm g_1) \in Ext~ B_{\mathbb{Y}^*}.\]
		From Lemma \ref{openmap}, it follows that there exist $u_1, u_2 \in \mathcal{A}_f$ such that $T u_1 \in \mathcal{B}_{g_1},T u_2 \in \mathcal{B}_{g_2}.$
		Now, $u_1, u_2 \in \mathcal{A}_f \implies u_1, u_2 \perp_B \ker f.$ This implies that $T(\ker f) \subset( T u_1)^{\perp_B^{\epsilon}} \cap (T u_2)^{\perp_B^{\epsilon}}.$ As $T$ is bijective so $T(\ker f)$ is a hyperspace of $\mathbb{Y}$ and  this contradicts the fact that $(Tu_1)^{\perp_B^{\epsilon}} \cap (T u_2)^{\perp_B^{\epsilon}}$ does not contain any hyperspace.  Thus, $T(\mathcal{A}_f) \subset \overline{\mathcal{B}}_g$ for some $g \in Ext~ B_{\mathbb{Y}^*}$.
		Now, we show that $T(\mathcal{A}_f) \subset \overline{\mathcal{B}}_g \implies T(\mathcal{A}_f) \subset \mathcal{B}_g$. On the contrary, suppose that there exists $x \in \mathcal{A}_f$ such that	$T x \in \overline{\mathcal{B}}_g \text { but } Tx \notin \mathcal{B}_g.$ Then there exists $h(\neq g) \in  Ext~ B_{\mathbb{Y}^*}$ such that $h \in J(Tx).$ So $Tx \in \overline{\mathcal{B}}_h$.
		From Lemma \ref{openmap}, it follows that there exist $v_1, v_2 \in \mathcal{A}_f$ such that $Tv_1 \in \mathcal{B}_g$ and $Tv_2 \in \mathcal{B}_h$. Again, $v_1, v_2 \in \mathcal{A}_f \implies v_1, v_2 \perp_B \ker f.$ This implies that $T(\ker f) \subset (T v_1)^{\perp_B^{\epsilon}} \cap (T v_2)^{\perp_B^{\epsilon}}$ and consequently this leads to a contradiction. Therefore, $T(\mathcal{A}_f) \subset {\mathcal{B}_g}$.\\
		Next, we show that for each $f\in Ext~B_{\mathbb{X}^*},$ there exists  a unique  $g\in Ext~B_{\mathbb{Y}^*}$ such that $T(\mathcal{A}_f) \subset {\mathcal{B}_g}$. If possible suppose that for some $f\in Ext~B_{\mathbb{X}^*},$ there exist two distinct  $g_1,g_2\in Ext~B_{\mathbb{Y}^*}$ such that $T(\mathcal{A}_f) \subset \mathcal{B}_{g_1}$ and  $T(\mathcal{A}_f) \subset \mathcal{B}_{g_2}.$ Then $T(\mathcal{A}_f) \subset \mathcal{B}_{g_1}\cap \mathcal{B}_{g_2}.$ Since  $g_1\neq g_2,$ it follows that $ \mathcal{B}_{g_1}\cap \mathcal{B}_{g_2}=\emptyset.$ This implies that $T(\mathcal{A}_f) \subset \emptyset,$ which is absurd as $\mathcal{A}_f\neq \emptyset$ and $T$ is bijective. Therefore, for each $f\in Ext~B_{\mathbb{X}^*}$ there exists  a unique  $g\in Ext~B_{\mathbb{Y}^*}$ such that $T(\mathcal{A}_f) \subset {\mathcal{B}_g}$.
	\end{proof}
	
	The following corollary is immediate from the above theorem.
	\begin{cor}
		Let $\mathbb{X},\mathbb{Y}$ be $n$-dimensional polyhedral Banach spaces. Let $0\leq \epsilon<\epsilon_{\mathbb{Y}}.$  If a non-zero operator $T\in \mathbb{L}(\mathbb{X},\mathbb{Y})$ preserves $\epsilon$-orthogonality at each $x \in \mathbb{X}$, then  $T$ maps each smooth point of $\mathbb{X}$ to a smooth point of $\mathbb{Y}.$
	\end{cor}
	In the following theorem, we investigate the connection between the number of facets of the unit balls of finite-dimensional polyhedral Banach spaces and the preservation of $\epsilon$-orthogonality by a linear operator between these spaces.
	\begin{theorem}\label{distinct}
		Let $\mathbb{X},\mathbb{Y}$ be $n$-dimensional polyhedral Banach spaces. For each $f \in Ext~ B_{\mathbb{X}^*}$, consider the set $\mathcal{A}_f=\{x \in \mathbb{X}: J(x)=\{f\}\}$ and  for each $g \in Ext~ B_{\mathbb{Y}^*}$, consider the set $\mathcal{B}_g=\{y \in \mathbb{Y}:J(y)=\{g\}\}.$  Let $0\leq \epsilon<\epsilon_{\mathbb{Y}}.$  If  a non-zero operator $T\in \mathbb{L}(\mathbb{X},\mathbb{Y})$ preserves $\epsilon$-orthogonality at each $x \in \mathbb{X},$ then the following results hold.
		\begin{itemize}
			\item[(i)] $|Ext~ B_{\mathbb{X}^*}|\geq|Ext~ B_{\mathbb{Y}^*}|.$ 
			\item[(ii)]  $|Ext~ B_{\mathbb{X}^*}|=|Ext~ B_{\mathbb{Y}^*}|$ if and only if  for any two distinct $f_1,f_2\in Ext~ B_{\mathbb{X}^*},$ there exist two distinct $g_1,g_2\in Ext~ B_{\mathbb{Y}^*} $ such that  $T(\mathcal{A}_{f_1}) \subset\mathcal{B}_{g_1}$ and $T(\mathcal{A}_{f_2}) \subset \mathcal{B}_{g_2}.$
		\end{itemize}
	\end{theorem}
	\begin{proof}
		Let  $T(\neq 0) \in \mathbb{L}(\mathbb{X},\mathbb{Y})$ preserve $\epsilon$-orthogonality at each $x \in \mathbb{X}.$ It follows from Lemma \ref{bijective} that $T$ is bijective. Since $\mathbb{X}$ and $\mathbb{Y}$ are finite-dimensional polyhedral Banach spaces, it follows that both $|Ext~ B_{\mathbb{X}^*}|$  and  $|Ext~ B_{\mathbb{Y}^*}|$ are finite.  
		
		(i)  Consider the mapping $\phi:Ext~B_{\mathbb{X}^*}\longrightarrow Ext~B_{\mathbb{Y}^*}$ defined by
		\[ \phi(f)=g,\text{	where } g\text{ is determined by the fact that }T(\mathcal{A}_f)\subset \mathcal{B}_g.\]
		As it follows from Theorem \ref{facetTofacet} that for each $f\in Ext~B_{\mathbb{X}^*}$ there exists a unique $g \in Ext~ B_{\mathbb{Y}^*}$ such that $T(\mathcal{A}_{f}) \subset \mathcal{B}_{g}, $ so the mapping $\phi$ is well defined. 
		We show that $\phi$ is surjective. Let  $\xi \in Ext~ B_{\mathbb{Y}^*}$ be arbitrary. As $\mathcal{B}_\xi\neq \emptyset,$ let $y(\neq0)\in \mathcal{B}_\xi.$ Since $T$ is bijective there exists $x(\neq 0)\in \mathbb{X}$ such that $Tx=y.$ As	$\bigcup\limits_{f \in Ext~ B_{\mathbb{X}^*}}\overline{ \mathcal{A}}_f=\mathbb{X},$ there exists $\zeta \in Ext~ B_{\mathbb{X}^*}$ such that $x\in \overline{\mathcal{A}}_\zeta.$
		Then it follows from Theorem \ref{facetTofacet} that there exists a unique $h \in Ext~ B_{\mathbb{Y}^*}$ such that $T(\mathcal{A}_{\zeta}) \subset \mathcal{B}_{h}.$
		Therefore, $Tx=y\in \overline{\mathcal{B}}_h\cap \mathcal{B}_{\xi}.$ If $h\neq \xi,$ then $\overline{\mathcal{B}}_h\cap \mathcal{B}_{\xi}=\emptyset,$ which is not possible. Thus, $\xi=h$ and so  $T(\mathcal{A}_{\zeta}) \subset \mathcal{B}_{\xi}.$ Hence 
		$\phi(\zeta)=\xi.$ Since $\xi$ is taken arbitrarily, it follows that  each element of $ Ext~ B_{\mathbb{Y}^*}$ has a pre-image in $Ext~ B_{\mathbb{X}^*}.$
		Therefore, $\phi$ is surjective and so $|Ext~ B_{\mathbb{X}^*}|\geq|Ext~ B_{\mathbb{Y}^*}|.$ This completes the proof of (i).
		
		(ii) As the mapping $\phi $ defined in (i) is surjective, the proof of (ii) follows easily from the hypothesis.
	\end{proof}
	
	Before going to the next result, we need to define the neighboring region of  the  set $\mathcal{A}_f=\{z \in \mathbb{X}:J(z)=\{f\}\}.$ For $f,g\in Ext~ B_{\mathbb{X}^*},$ the set $\mathcal A_g$ is said to be a neighboring region of  $\mathcal{A}_{f}$  if $\overline{ \mathcal{A}}_{f}\cap \overline{ \mathcal{A}}_{g}\neq \{0\}.$ For  each $f\in Ext~ B_{\mathbb{X}^*},$ we denote the set of all neighboring regions of $\mathcal{A}_f$ by $\tau(\mathcal{A}_f)$ and thus, 
	\[\tau(\mathcal{A}_f)=\{\mathcal{A}_h:\overline{ \mathcal{A}}_{f}\cap \overline{ \mathcal{A}}_{h}\neq \{0\}\}.\]
	Our next result is as follows. 
	\begin{theorem}\label{cardpreserve1}
		Let $\mathbb{X},\mathbb{Y}$ be $n$-dimensional polyhedral Banach spaces such that $|Ext~ B_{\mathbb{X}^*}|$ $=|Ext~ B_{\mathbb{Y}^*}|.$   For each $f \in Ext~ B_{\mathbb{X}^*}$, consider the set $\mathcal{A}_f=\{z \in \mathbb{X}: J(z)=\{f\}\}$ and for each $g \in Ext~ B_{\mathbb{Y}^*}$, consider the set $\mathcal{B}_g=\{z \in \mathbb{X}: J(z)=\{g\}\}.$ 
		Let $0\leq \epsilon <\epsilon_{\mathbb{Y}}.$ If $T \in \mathbb{L}(\mathbb{X},\mathbb{Y})$ preserves $\epsilon$-orthogonality at each $x \in \mathbb{X},$ then the following results hold true.
		\begin{itemize} \label{card fac equ}
			\item[(i)]	For each $f \in  Ext~ B_{\mathbb{X}^*},$ there exists a unique  $g \in  Ext~ B_{\mathbb{Y}^*}$ such that  $T(\mathcal{A}_f) = \mathcal{B}_g.$ 
			\item[(ii)] For each non-zero $x\in \mathbb{X},$  $|Ext ~J(x)| =|Ext ~J(T x)|$.
			\item[(iii)] For each $f\in Ext B_{\mathbb{X}^*},$ $|\tau(\mathcal{A}_f)|=|\tau(\mathcal{B}_g)|,$ where $T(\mathcal{A}_f)\subset \mathcal{B}_g, ~g\in Ext~ B_{\mathbb{Y}^*}.$
		\end{itemize}
	\end{theorem}
	\begin{proof} 
		(i)  	Let $f\in Ext~B_{\mathbb{X}^*}.$ From Theorem \ref{facetTofacet}, it follows that  $T(\mathcal{A}_{f}) \subset \mathcal{B}_{g}, $ for a unique $g \in Ext~ B_{\mathbb{Y}^*}.$ We show that   $T(\mathcal{A}_{f}) = \mathcal{B}_{g}. $ On the contrary suppose that $y\in \mathcal{B}_{g}\setminus T(\mathcal{A}_{f}).$  From Lemma \ref{bijective}, it follows that $T$ is bijective and so there exists $x(\neq 0)\in \mathbb{X}$ such that $Tx=y.$ Then there exists $f_1(\neq f)\in Ext~B_{\mathbb{X}^*} $ such that $f_1\in Ext~J(x).$ Thus,  $x\in \overline{\mathcal{A}}_{f_1}.$  From [(ii), Theorem \ref{distinct}], it follows that there exists  $g_1(\neq g) \in Ext~ B_{\mathbb{Y}^*}$ such that $T(\mathcal{A}_{f_1}) \subset \mathcal{B}_{g_1}.$ Thus, $Tx\in\overline{\mathcal{B}}_{g_1} $ and so $Tx=y\in \overline{\mathcal{B}}_{g_1}\cap\mathcal{B}_{g}=\emptyset.$ This is a contradiction. Thus, $T(\mathcal{A}_{f}) = \mathcal{B}_{g}. $

		(ii) 	Let $x(\neq 0) \in \mathbb{X}$ and let $f\in Ext~J(x)$ be arbitrary. From Proposition \ref{propo}, it follows that $x\in \overline{ \mathcal{A}}_{f}.$ Now, it follows from (i) that there exists a unique $g \in Ext~ B_{\mathbb{Y}^*}$ such that $T(\mathcal{A}_{f}) = \mathcal{B}_{g}.$ Thus, $Tx\in  \overline{ \mathcal{B}}_{g}$ and so $g\in Ext~J(Tx).$  Therefore, for each $f \in Ext~ J(x),$ there exists a unique $g\in Ext~J(Tx)$ such that  $T(\mathcal{A}_f)= \mathcal{B}_g.$
		Consider the mapping $\psi:Ext~J(x)\longrightarrow Ext~J(Tx)$ defined by
		\[ \psi(f)=g,\text{	where } g\text{ is determined by}~ T(\mathcal{A}_f)= \mathcal{B}_g.\]
		Since  $|Ext~ B_{\mathbb{X}^*}|$
		$=|Ext~ B_{\mathbb{Y}^*}|,$ it follows from \ref{distinct} that $\psi $ is injective. Next, we show that $\psi$ is surjective. Let $\xi\in Ext~J(Tx).$ Then there exists $\zeta\in Ext~ B_{\mathbb{X}^*}$ such that $T(\mathcal{A}_{\zeta})= \mathcal{B}_{\xi}.$ Clearly, $Tx\in \overline{ \mathcal{B}}_{\xi}.$ Let a sequence $\{x_n\}\subset \mathcal{A}_{\zeta}$ be such that $Tx_n\longrightarrow Tx.$ Since $\mathbb{X}$ is finite-dimensional and $T$ is bijective, $T^{-1}$ exists and is continuous. Hence $T^{-1}(Tx_n)\longrightarrow  T^{-1}(Tx)\implies x_n \longrightarrow  x$ and so $x\in \mathcal{A}_{\zeta}.$ From Proposition \ref{propo}, it follows that $\zeta\in Ext~J(x).$ Thus, there exists $\zeta\in Ext~J(x)$ such that $\psi(\zeta)=\xi.$ Hence $\psi$ is surjective and consequently, $\psi$ is bijective.  Therefore, $| Ext~J(x)|=|Ext~J(T x)|.$
		
		(iii)   Let $T(\neq 0) \in \mathbb{L}(\mathbb{X},\mathbb{Y})$ preserve $\epsilon$-orthogonality at each $x \in \mathbb{X}.$ From Lemma \ref{bijective}, it follows that $T$ is bijective. Let  $ \mathcal{A}_{\eta} \in\tau(\mathcal{A}_f).$ Then from (i), it follows that there exists a unique $\zeta\in Ext~ B_{\mathbb{Y}^*}$ such that $T(\mathcal{A}_{\eta})= \mathcal{B}_{\zeta}.$ We claim that $\mathcal{B}_{\zeta}\in\tau(\mathcal{B}_g).$   Let $ x(\neq 0)\in \overline{ \mathcal{A}}_{f}\cap \overline{ \mathcal{A}}_{\eta}$ and this implies that $Tx\in T(\overline{ \mathcal{A}}_{f})\cap T(\overline{ \mathcal{A}}_{\eta}).$ Thus, $Tx(\neq 0)\in \overline{ \mathcal{B}}_{g}\cap \overline {\mathcal{B}}_{\zeta}$ and so  $\mathcal{B}_{\zeta}\in \tau(\mathcal{B}_g).$ Thus our claim is established. Consider $\phi:\tau(\mathcal{A}_f)\longrightarrow \tau(\mathcal{B}_g)$ defined by
		\[ \phi(\mathcal{A}_\eta)=\mathcal{B}_\zeta.\]
		From the above discussion, it follows that  $\phi$ is well-defined.\\
		Since $T$ is bijective, it follows that $\phi$ is injective. Now, we show that $\phi$ is surjective. Let  $\mathcal{B}_{g_1}\in \tau(\mathcal{B}_g).$ Then there exists   $f_1(\neq f)\in Ext~ B_{\mathbb{X}^*}$ such that $T(\mathcal{A}_{f_1})= \mathcal{B}_{g_1}.$
		Let $y(\neq0)\in\overline {\mathcal{B}}_{g}\cap \overline {\mathcal{B}}_{g_1}.$ Since $T$ is bijective, there exists $x\in \mathbb{X}$ such that $Tx=y.$ Let  $\{u_n\}\subset \mathcal{A}_{f}$ and $\{v_n\}\subset \mathcal{A}_{f_1}$ be such that $Tu_n\longrightarrow Tx$ and $Tv_n\longrightarrow Tx.$ As $\mathbb{X}$ is finite-dimensional and $T$ is bijective, $T^{-1}$ exists and continuous and therefore, $T^{-1}(Tu_n)\longrightarrow T^{-1}(Tx)\implies u_n \longrightarrow x.$ Thus, $x\in \overline{ \mathcal{A}}_{f}.$ Again,  $T^{-1}(Tv_n)\longrightarrow T^{-1}(Tx)\implies v_n \longrightarrow x$ and so $x\in \overline{ \mathcal{A}}_{f_1}.$ Thus, $x\in \overline A_{f}\cap \overline{ \mathcal{A}}_{f_1}$ and so
		$\mathcal{A}_{f_1}\in\tau(\mathcal{A}_f).$  Therefore, $\phi$ is surjective and consequently  $\phi$ is bijective. This completes the proof.
	\end{proof}
	
	We note that for polyhedral Banach spaces $\mathbb{X}$ and $\mathbb{Y} $ with $|Ext~ B_{\mathbb{X}^*}|>|Ext~ B_{\mathbb{Y}^*}|,$  there may exist an operator  $T\in\mathbb{L}(\mathbb{X}, \mathbb{Y})$ that preserves $\epsilon$-orthogonality for some $0<\epsilon<\epsilon_{\mathbb{Y}}$ at each $x\in \mathbb{X}.$ The following example illustrates the scenario. 
	\begin{example}\label{exnotequal}
		Let  $\mathbb{X}$ be the two-dimensional polyhedral Banach space whose unit sphere is the regular octagon  with vertices $\{u_i:1\leq i\leq8\},$ where
		$$u_i=\Big(\cos(i-1)\frac{\pi}{4},\sin(i-1)\frac{\pi}{4}\Big).$$
		and let $\mathbb{Y}=\ell_{\infty}^2.$
		Let $T\in\mathbb{L}(\mathbb{X}, \mathbb{Y})$ is given by $T(x,y)=(x-y, x+y).$ The situation is illustrated in the following figure.
		
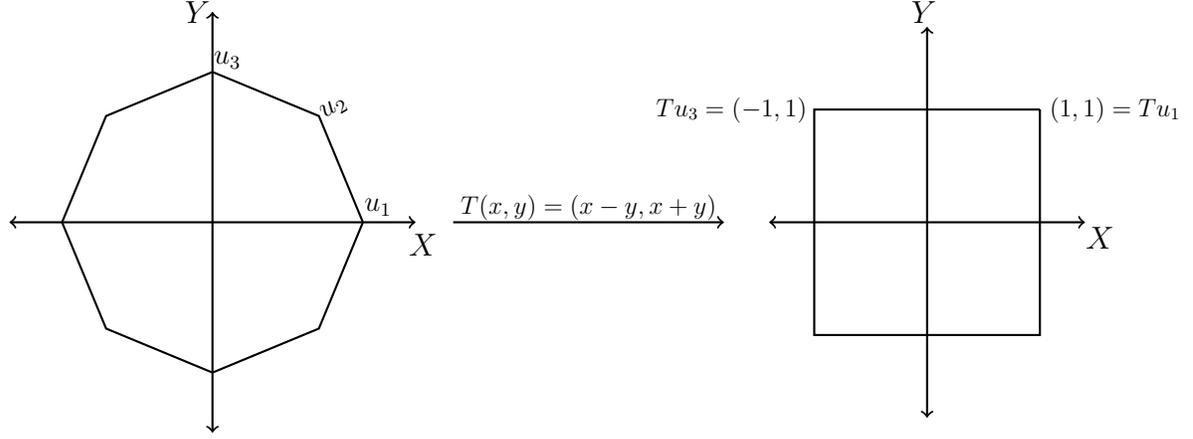
\[
		\begin{tikzpicture}[scale=2]
			
			\draw[black,thick](1,0)--({1/sqrt(2)},{1/sqrt(2)})--(0,1)--({-1/sqrt(2)},{1/sqrt(2)})--(-1,0)--({-1/sqrt(2)},{-1/sqrt(2)})--(0,-1)--({1/sqrt(2)},{-1/sqrt(2)})--(1,0);
			\draw[black,thick](5.5,.75)--(4,.75)--(4,-.75)--(5.5,-.75)--(5.5,.75);
			\node [black,scale=.9] at (1.1,.1) {$u_1$};
			\node [black,rotate=30,scale=.9] at (({1/sqrt(2)+.1},{1/sqrt(2)+.05}) {$u_2$};
			\node [black,scale=.8] at (6,.75) {$(1,1)=Tu_1$};	
			\node [black,scale=.9] at (.1,1.08) {$u_3$};	
			\node [black,scale=.8] at (3.45,.75) {$Tu_3=(-1,1)$};
			\draw[black,thick,->](1.6,0)--(3.4,0);
			\node [black,scale=.8] at (2.5,.1) {$T(x,y)=(x-y, x+y)$};
			\draw[<->, black,thick ](-1.35,0)--(1.35,0);
			\draw[<->, black,thick ](3.7,0)--(5.8,0);
			\draw[<->, black,thick ](0,-1.4)--(0,1.4);	
			\draw[<->, black,thick ](4.75,-1.3)--(4.75,1.3);	
			\node [black] at (1.4, -.15) {$X$};	
			\node [black] at (5.9, -.1) {$X$};
			\node [black] at (-.1, 1.4) {$Y$};	
			\node [black] at (4.73, 1.4) {$Y$};
		\end{tikzpicture}
		\]
		
		Let $\sqrt{2}-1<\epsilon<1.$ We show that $T$ preserves $\epsilon$-orthogonality at each point of $\mathbb{X}.$  Since Birkhoff-James orthogonality is homogeneous, it is sufficient to show that $T$ preserves $\epsilon$-orthogonality at each point of $S_{\mathbb{X}}.$ \\
		First we show that $T$  preserves $\epsilon$-orthogonality at each point of the facet $co(\{u_1,u_2\}).$ Let $u\in co(\{u_1,u_2\})$ and so $u=(1-\lambda)u_1+\lambda u_2$ for some $\lambda\in [0,1].$\\
		\textbf{Case 1:} Let $\lambda=0.$ Then $u=u_1$ and  $u^{\perp_B}=\big\{rw:r\in\mathbb{R} \text{ and } w=(1-t)\frac{u_2+u_3}{2} +t \frac{u_3+u_4}{2}, t\in [0,1]\big\}.$ Let $t\in [0,1]$ be arbitrary and let $w_1= (1-t)\frac{u_2+u_3}{2} +t \frac{u_3+u_4}{2}.$ Then $Tw_1=\big(-\frac{1}{2}-\frac{t}{\sqrt 2}, \frac{1}{2}+\frac{1-t}{\sqrt 2}\big).$ It is easy to observe that $Tu\perp_B Tw_1$ and so $ Tu\perp_B T(rw_1)$ for all $r\in \mathbb {R}.$  Since $t$ is taken arbitrarily, we can say that $T$ preserves Birkhoff-James orthogonality at  $u$ and so $T$ preserves $\epsilon$-orthogonality at $u.$ \\
		\textbf{Case 2:} Let $\lambda=1.$ Then $u=u_2$ and  $u^{\perp_B}=\big\{rw:r\in\mathbb{R} \text{ and } w=(1-t)\frac{u_3+u_4}{2} +t \frac{u_4+u_5}{2}, t\in [0,1]\big\}.$ Let $t\in [0,1]$ be arbitrary and let $w_1=(1-t)\frac{u_3+u_4}{2} +t \frac{u_4+u_5}{2}.$ Then $Tw_1=\big(-\frac{1}{2}-\frac{1}{\sqrt 2}, \frac{1}{2}-t\big).$ Let $f(\alpha, \beta)=\beta.$ Then $f\in J(Tu).$ Now, $|f(Tw_1)|=|\frac{1}{2}-t|\leq \frac{1}{2}\leq \epsilon \max\big\{|\frac12+\frac1{\sqrt{2}}|,|\frac{1}{2}-t|\big\}=\epsilon\|Tw_1\|.$ Clearly, $|f(T(rw_1))|\leq \epsilon\|T(rw_1)\|$ for all $r\in \mathbb {R}.$    Since $t$ is taken arbitrarily, from Theorem \ref{charac}, we can say that  $T$ preserves $\epsilon$-orthogonality at $u.$ \\
		\textbf{Case 3:} Let $0<\lambda<1.$  Then $Tu=\big(1-\lambda,1-\lambda+\sqrt 2\lambda\big).$ It is easy to see that $u^{\perp_B}\cap S_{\mathbb{X}}=\{ \frac{u_3+u_4}{2}\}.$ Let $f(\alpha, \beta)=\beta.$ Then $f\in J(Tu).$  Let $w=\frac{u_3+u_4}{2}.$ Then $Tw=\big(-\frac{1}{2}-\frac{1}{\sqrt{2}},\frac{1}{2}\big).$ Now, $|f(Tw)|=\frac{1}{2}\leq  \epsilon \max\big\{|\frac{1}{2}+\frac1{\sqrt{2}}|,\frac{1}{2}\big\}=\epsilon\|Tw\|.$ So $T$ preserves $\epsilon$-orthogonality at $u.$\\
		Thus, $T$ preserves $\epsilon$-orthogonality at each point of $co(\{u_1,u_2\}).$ By similar arguments, we can say that for each $i=2,3,\dots,8,$ $T$ preserves  $\epsilon$-orthogonality at each point of the facet $co(\{u_i,u_{i+1}\}) (u_9=u_1).$  Therefore, $T$ preserves $\epsilon$-orthogonality at each point of $S_{\mathbb{X}}$ and so  $T$ preserves $\epsilon$-orthogonality at each point of $\mathbb{X}.$
	\end{example}
	\begin{remark}\label{rem facet}
		Let $\mathbb{X},\mathbb{Y}$ be $n$-dimensional polyhedral Banach spaces.  
		From Theorem \ref{facetTofacet}, Theorem \ref{distinct} and the above Example \ref{exnotequal}, we have the following observations:
		\begin{enumerate}
			\item[(i)] If $T\in\mathbb{L}(\mathbb{X},\mathbb{Y})$ preserves $\epsilon$-orthogonality for some $0<\epsilon<\epsilon_{\mathbb{Y}}$ at each point of $\mathbb{X},$ then the number of facets of $B_{\mathbb{Y}}$ can not be greater than the number of facets of $B_{\mathbb{X}}.$
			
			\item[(ii)] Let $T\in\mathbb{L}(\mathbb{X},\mathbb{Y})$ preserve $\epsilon$-orthogonality for some $0<\epsilon<\epsilon_{\mathbb{Y}}$ at each point of $\mathbb{X}.$ Then the images of two different  facets of $B_{\mathbb{X}}$ may be contained in the conical hull of a facet of $B_{\mathbb{X}}.$ 
			
			\item[(iii)] In Example \ref{exnotequal}, we see that $u_2=(\frac{1}{\sqrt{2}},\frac{1}{\sqrt{2}})$ is a $2$-smooth point of $\mathbb{X}$ but $Tu_2=(0,\frac {2} {\sqrt 2})$ is a smooth point of $\mathbb{Y}.$ So we conclude that:\\
			If $T\in\mathbb{L}(\mathbb{X},\mathbb{Y})$ preserves $\epsilon$-orthogonality for some $0<\epsilon<\epsilon_{\mathbb{Y}}$ at each point of $\mathbb{X},$ then the image of a non-smooth point of $\mathbb{X}$ may be be a smooth point of $\mathbb{Y}.$ \\
			Moreover, for some $x\in\mathbb{X},$ cardinality of $Ext ~J(T x)$ may be less than the cardinality of $Ext ~J(x).$
		\end{enumerate}
	\end{remark}
	
	In contrast to [(iii), Remark \ref{rem facet} ], when the  domain and range are the same space, we have the following corollary, which is immediate from Theorem \ref{card fac equ}.
	\begin{cor}\label{non smooth}
		Let $\mathbb{X}$ be a finite-dimensional polyhedral Banach  space. Let $0\leq \epsilon<\epsilon_{\mathbb{X}}.$  If  a non-zero operator $T \in \mathbb{L}(\mathbb{X})$ preserves $\epsilon$-orthogonality at each $x \in \mathbb{X}$, then  $T$ maps each non-smooth point of $\mathbb{X}$ to a non-smooth point of $\mathbb{X}.$
	\end{cor} 
	Further, under an additional assumption on the polyhedral Banach space, we have the following corollary, the proof of which follows easily from Theorem \ref{card fac equ}.
	\begin{cor}\label{corsmoothpreserve}
		Let $\mathbb{X}$ be a finite-dimensional polyhedral Banach  space such that  for any two $x_1, x_2 \in \mathbb{X},$ $|Ext~J(x_1)|>|Ext~J(x_2)|$ whenever the order of smoothness of $x_1$ is greater than the order of smoothness of $x_2.$ Suppose  $0\leq \epsilon <\epsilon_{\mathbb{X}}.$ If  a non-zero operator $T\in \mathbb{L}(\mathbb{X})$ preserves $\epsilon$-orthogonality at each $x \in \mathbb{X}$, then for each $z\in \mathbb{X},$  the order of smoothness of $Tz$ is equal to the order of smoothness of $z.$
	\end{cor}
	Next, we give an example to show that the above result does not hold locally.
	\begin{example}
		Let $\mathbb{X}=\ell_{\infty}^3$ and let $0<\epsilon<1.$ Define $T\in \mathbb{L}(\mathbb{X})$ by $T(x,y,z)=(\epsilon x,x-y,y-z)$ for $(x,y,z)\in \mathbb{X}.$ Let $u=(1,1,1).$ Then $Tu=(\epsilon,0,0).$ Clearly, $3=| Ext~J(u)|>|Ext~J(T u)|=1.$
		Now, we show that $T$ preserves $\epsilon$-orthogonality at $u.$ Note that $u^{\perp_B}=\{(a,b,c)\in\mathbb{X}:ab\leq0\text{ or }bc\leq0 \text{ or }ac\leq0\}.$ Let $v=(a,b,c)\in u^{\perp_B}.$ Then $Tv=(\epsilon a,a-b,b-c).$ Let $f\in \mathbb{X}^*$ be such that $f(x,y,z)=x$ for all $(x,y,z)\in \mathbb{X}.$ Then $f\in J(Tu).$ Next, $|f(Tv)|=|a\epsilon|\leq \epsilon \max\{|\epsilon a|,|a-b|,|b-c|\}=\epsilon \|Tv\|.$ So $T$ preserves $\epsilon$-orthogonality at $u.$
	\end{example}

	In case the cardinality of the set of all extreme supporting functionals at a point is at least $ 3 $, we have the following  observation.
	\begin{proposition}\label{2smooth}
		Let $\mathbb{X}$ be a normed linear space. If $|Ext~J(x)|\geq 3,$ then any three distinct $f_1, f_2, f_3\in Ext~J(x) $ are linearly independent.
	\end{proposition}
	\begin{proof}
		Let $x\in\mathbb{X}$ be such that $|Ext~J(x)|\geq 3.$ Consider three distinct support functionals $f_1, f_2, f_3\in Ext~J(x).$ If possible suppose that $f_3=af_1+bf_2$ for some $a,b\in\mathbb{R}\setminus \{0\}.$ Now,
		\[f_3(x)=(af_1+bf_2)(x)\implies \|x\|=a\|x\|+b\|x\|\implies a+b=1.\]
		Since $f_3\in Ext~J(x)$, it follows that $ab<0.$ Without loss of generality, we may assume that $a>0 $ and $b<0.$ Now, $\frac{1}{a}f_3+\frac{-b}{a}f_2=f_1.$ This shows that $\frac{1}{a}>0,~\frac{-b}{a}>0$ and $\frac{1}{a}+\frac{-b}{a}=1.$ This contradicts the fact that $ f_1\in Ext~J(x).$ Therefore, $f_1,f_2,f_3$ are linearly independent.
	\end{proof}
	
	\begin{remark}
		Using  Theorem \ref{cardpreserve1} and Proposition \ref{2smooth}, we observe that for  any finite-dimensional polyhedral Banach space  $\mathbb{X}$ and $0\leq\epsilon<\epsilon_{\mathbb{X}},$ if  a non-zero operator $T\in \mathbb{L} (\mathbb{X})$ preserves $\epsilon$-orthogonality at each $x\in \mathbb{X},$ then the image of each $2$-smooth point of $\mathbb{X}$ is a $2$-smooth point of $\mathbb{X}.$ 
	\end{remark}
	
	Our next result gives a partial generalization of the Blanco-Koldobsky-Turn\v{s}ek theorem in case of  some finite-dimensional polyhedral Banach spaces under an additional assumption on the norm of the images of the extreme points of the unit ball of the domain. 
	\begin{theorem}\label{isometry}
		Let $\mathbb{X}$ be a finite-dimensional polyhedral Banach space such that for each $x\in Ext~B_{\mathbb{X}},$ $|Ext~J(x)|>|Ext~J(y)|$ for all non-zero $y\in \mathbb{X}\setminus Ext~B_{\mathbb{X}}.$ Then a norm one operator $T\in \mathbb{L}(\mathbb{X})$ is an isometry if and only if $T$ satisfies the following conditions: 
		\begin{itemize}
			\item[(i)]  $\| Tu\|=\|Tv\|$ for any $u,v \in Ext~B_\mathbb{X}.$
			\item[(ii)] $T$ preserves $\epsilon$-orthogonality for any $0\leq\epsilon<\epsilon_{\mathbb{X}}$ at each $x\in \mathbb{X}.$  
		\end{itemize} 
	\end{theorem}
	\begin{proof}
		We only prove the sufficient part as the necessary part is trivial.
		Let $0\leq \epsilon< \epsilon_{\mathbb{X}}$ and let  $T\in \mathbb{L}(\mathbb{X})$
		preserve $\epsilon$-orthogonality at each $x\in\mathbb{X}.$ Then it follows from [(ii), Theorem \ref{card fac equ}]  that $|Ext~J(x)|=|Ext~J(Tx)|$ for all $x(\neq0)\in\mathbb{X}$ 
		and so from the hypothesis it follows that $T$ maps each extreme point of $B_{\mathbb{X}}$  to a scalar multiple of some extreme point of $B_{\mathbb{X}}.$ Since $\mathbb{X}$ is finite-dimensional, $T$ must attains its norm at an  extreme point of $B_\mathbb{X}$ and so  by the condition (i), $\| Tu\|=\|T\|=1$ for all extreme points $u$ of $B_\mathbb{X}.$ Therefore, $T$ maps each extreme point of $B_\mathbb{X} $ to some extreme point of $B_\mathbb{X}.$  From Lemma \ref{bijective}, it follows that $T$ is bijective and so $T^{-1}$ exists and maps each extreme point of $B_\mathbb{X} $ to some extreme point of $B_\mathbb{X}.$ Now, $T^{-1}$ also attains its norm at some extreme point of $B_\mathbb{X}$ and so $\|T^{-1}\|=1.$ 
		Hence  $\|Tx\|\leq \|x\|=\|T^{-1}(Tx)\|\leq \|Tx\|$ for all $x\in\mathbb{X}.$  This implies $\|Tx\|=\|x\|$ for all  $x\in\mathbb{X}.$ Therefore, $T$ is an isometry. 
	\end{proof}
	For a two-dimensional polyhedral Banach space whose unit sphere is a $2n$-gon, we have the following observation.
	\begin{lemma}\label{conse}
		Let $ \mathbb{X} $  be a two-dimensional polyhedral Banach space whose unit sphere is a $2n$-gon for some $n\in\mathbb{N}.$ Let $0\leq \epsilon<\epsilon_{\mathbb{X}}.$  If  a non-zero operator $T \in \mathbb{L}(\mathbb{X})$ preserves $\epsilon$-orthogonality at each $x \in \mathbb{X}$ then $T$ maps two consecutive extreme points of $B_\mathbb{X} $ to scalar multiples of two consecutive extreme points of $B_\mathbb{X}.$ 
	\end{lemma}
	
	\begin{proof}
		Let $T (\neq 0)\in \mathbb{L}(\mathbb{X})$ preserve $\epsilon$-orthogonality at each $x \in \mathbb{X}.$ Therefore, it follows from Corollary \ref{non smooth} that the image of each extreme point of $B_\mathbb{X} $ is a scalar multiple of some extreme point of $B_\mathbb{X}.$ Applying Lemma \ref{bijective}, we conclude that $T$ is bijective. Let $u_1, u_2$ be two consecutive extreme points of $B_\mathbb{X}.$ If possible suppose that $Tu_1=k_1v_i$ and $Tu_2=k_2v_j,$ where $v_i,v_j$ are two non-consecutive extreme points of $B_\mathbb{X} $ and $k_1,k_2>0.$ Since $v_i, v_j$ are two non-consecutive extreme points of $B_\mathbb{X},$ there exists at least one extreme point $v$ of $B_\mathbb{X} $ such that $v=\alpha (1-t)v_i+\alpha  t v_j$ for some $\alpha>0$ and $t\in(0,1).$ Since $T$ is bijective and the image of  each extreme point of $B_\mathbb{X} $ is a scalar multiple of some extreme point of $B_\mathbb{X},$ it follows that  $T^{-1}v$ is a scalar multiple of an extreme point of $B_\mathbb{X}.$  Now, $T^{-1}v=T^{-1}\big(\alpha (1-t)v_i+\alpha  t v_j\big)=(1-t)\frac{\alpha}{k_1}u_1+t \frac{\alpha}{k_2} u_2.$ 
		This contradicts the hypothesis that $u_1,u_2$ are two consecutive extreme points of $B_\mathbb{X}.$ Hence the images of any two consecutive extreme points of $B_\mathbb{X}$ are scalar multiples of two consecutive extreme points of $B_\mathbb{X}.$
	\end{proof}
	Next, we prove that  a two-dimensional Banach space whose unit sphere is a regular polygon has the \textit{Property P}.
	\begin{theorem}\label{regular}
		Let $\mathbb{X}$ be the two-dimensional polyhedral Banach space whose unit sphere is the  regular $2n$-gon, for some $n\in\mathbb{N}\setminus\{1,2\},$ with vertices $$v_j=\left(\cos(j-1)\frac{\pi}{n},\sin(j-1)\frac{\pi}{n}\right),~ 1\leq j\leq 2n.$$ Then $\mathbb{X}$  satisfies the \textit{Property P}.
	\end{theorem}
	\begin{proof}
		We show that  any operator $T\in \mathbb{L}(\mathbb{X})$ preserves $\epsilon$-orthogonality for  any $0< \epsilon<\epsilon_{\mathbb{X}}$ at each $x\in \mathbb{X}$ if and only if $T$ is a scalar multiple of an isometry.\\
		Sufficient part of the above statement follows trivially. For the necessary part, let $0< \epsilon< \epsilon_{\mathbb{X}}$ and let  $T\in \mathbb{L}(\mathbb{X})$
		preserve $\epsilon$-orthogonality at each $x\in \mathbb{X}.$  If $T=0,$ we have nothing to prove. Let $T\neq 0.$ From Lemma \ref{bijective}, it follows that $T$ is bijective.  Suppose $v_i,v_{i+1},v_{i+2}$ are three consecutive vertices of $S_\mathbb{X}$ for some $1\leq i\leq 2n.$ We note that $v_{2n+r}=v_r$ for any $1\leq r\leq 2n.$ Now, by a simple calculation it can be seen that $v_{i+2}=-v_i+2 \cos\Big(\frac{\pi}{n}\Big) v_{i+1}.$ Next, it follows from Lemma  \ref{conse} that  $T$ maps any two consecutive vertices of $S_\mathbb{X} $ to scalar multiples of two consecutive vertices of $S_\mathbb{X}.$ So $Tv_i=k_1v_j,Tv_{i+1}=k_2v_{j+1},Tv_{i+2}=k_3v_{j+2}$ for some $k_1,k_2,k_3\in \mathbb R\setminus\{0\} $ and $1\leq j\leq 2n .$  Then 
		\begin{eqnarray*}
			&&k_3v_{j+2}=-k_1v_j+2k_2 \cos\Big(\frac{\pi}{n}\Big) v_{j+1}\\
			&\implies&k_3\Big(-v_j+2 \cos\Big(\frac{\pi}{n}\Big) v_{j+1}\Big)=-k_1v_j+2k_2 \cos\Big(\frac{\pi}{n}\Big) v_{j+1}\\
			&\implies& \Big(1-\frac{k_1}{k_3}\Big)v_j=2\Big(1-\frac{k_2}{k_3}\Big)\cos\Big(\frac{\pi}{n}\Big) v_{j+1}.
		\end{eqnarray*}
		Since $v_j$ and $v_{j+1}$ are linearly independent and $n\neq2,$ it follows that $k_1=k_2=k_3.$ Thus, $\|Tv_i\|=\|Tv_{i+1}\|$ for any $1\leq i\leq 2n-1.$ Therefore, from Theorem \ref{isometry}, it follows that $T$ is an isometry multiplied by a constant.
	\end{proof}

	Next, we prove that $\ell_{\infty}^n~(n>2) $ has the \textit{Property P}. Moreover,  we show that  every bounded linear operator on $\ell_{\infty}^n$  preserving  $\epsilon$-orthogonality  for each $0\leq \epsilon <1$ is an isometry multiplied by a constant.
	\begin{theorem}\label{ell inf}
		Let $\mathbb{X}=\ell_{\infty}^n~(n>2)$ and let $0\leq\epsilon <1.$ Then	 $T\in\mathbb{L}(\mathbb{X})$ preserves $\epsilon$-orthogonality at each $x\in\mathbb{X}$ if and only if $T$ is a scalar multiple of an isometry.
	\end{theorem}
	\begin{proof}
		Sufficient part of this theorem follows easily.\\
		For necessary part, let $0\leq\epsilon <1.$ If $T=0$ then we have nothing to show. Let $T\neq0.$\\
		Now, let us consider 
		$$A=\left\{u_i=\big(u_i(1), u_i(2),\ldots ,u_i(n)\big): 1\leq i\leq n\right\},$$
		where  for each $1\leq i,j\leq n,$ 
		\[ u_i(j) =  \begin{cases}
			1,& \text{ if } j=1\\
			-1,&\text{ if }1<j\leq i\\
			1,& \text{ if }j>i.
		\end{cases} \]
		
		It is easy to see that $A\subset Ext ~B_{\mathbb{X}}$ and $A$ forms a basis for $\mathbb{X}.$ Let $T\in\mathbb{L}(\mathbb{X})$ be such that  $T$ preserves $\epsilon$-orthogonality at each $x\in\mathbb{X}.$ If $T=0,$ then the theorem is trivially true. Suppose that $T\neq 0.$  From Lemma \ref{bijective}, It follows that $T $ is bijective. Next, it follows from Corollary \ref{corsmoothpreserve} that for each $1\leq i\leq n$,
		\[T(u_i)= k_iv_i,~ k_i \neq 0,~ v_i \in Ext ~B_{\mathbb{X}} .\]
		
		Now, consider the set 
		$$B=\left\{w_i=\big(w_i(1), w_i(2),\ldots ,w_i(n)\big): 1\leq i\leq n-1\right\},$$
		where  for each $ 1\leq i\leq n-1,~ 1\leq j\leq n.$
		\[ w_i(j) =  \begin{cases}
			1,& \text{ if }j\leq n-i\\
			-1,& \text{ if }j>n-i.
		\end{cases} \]
		Then by a simple calculation, it can be seen that for each $1\leq i <n-1, ~~w_i= u_1-u_{n-i}+u_n$ and so $T(w_i)= k_1v_1-k_{n-i}v_{n-i}+k_nv_n.$\\
		Since for each $1\leq i<n-1, ~ v_1 , v_{n-i},v_n$ are linearly independent, we have the following four cases:\\
		\textbf{Case 1:} There exist $1\leq j_1, j_2, j_3 \leq n$ such that
		\[	-sgn(v_1(j_1))= sgn (v_{n-i}(j_1)) = sgn (v_n(j_1)),\]
		\[ sgn (v_1(j_2))= -sgn (v_{n-i}(j_2)) = sgn (v_n(j_2)),\]
		\[	sgn (v_1(j_3))= sgn (v_{n-i}(j_3)) = -sgn (v_n(j_3)).\]
		\textbf{Case 2:} There exist $1\leq j_1, j_2, j_3 \leq n$ such that
		\[	sgn (v_1(j_1))= sgn (v_{n-i}(j_1)) = sgn (v_n(j_1)),\]
		\[ sgn (v_1(j_2))= -sgn (v_{n-i}(j_2)) = sgn (v_n(j_2)),\]
		\[	sgn (v_1(j_3))= sgn (v_{n-i}(j_3)) = -sgn (v_n(j_3)).\] 
		\textbf{Case 3:} There exist $1\leq j_1, j_2, j_3 \leq n$ such that
		\[	sgn (v_1(j_1))= sgn (v_{n-i}(j_1)) = sgn (v_n(j_1)),\]
		\[ -sgn (v_1(j_2))= sgn (v_{n-i}(j_2)) = sgn (v_n(j_2)),\]
		\[	sgn(v_1(j_3))= sgn( v_{n-i}(j_3)) = -sgn (v_n(j_3)).\]
		\textbf{Case 4:} There exist $1\leq j_1, j_2, j_3 \leq n$ such that
		\[	sgn (v_1(j_1))= sgn (v_{n-i}(j_1)) = sgn (v_n(j_1)),\]
		\[ -sgn (v_1(j_2))= sgn (v_{n-i}(j_2)) = sgn (v_n(j_2)),\]
		\[	sgn (v_1(j_3))= -sgn(v_{n-i}(j_3)) = sgn (v_n(j_3)).\]
		Now, if \textbf{Case 1} holds true, then 
		\[|-k_1-k_{n-i}+k_n|=|k_1+k_{n-i}+k_n|= |k_1-k_{n-i}-k_n|.\]
		This shows that \[|k_1|=|k_{n-i}|=|k_n|.\]
		Similarly, if \textbf{Case 2, 3, 4} hold true, then 
		\[|k_1|=|k_{n-i}|=|k_n|.\]
		Since this is true for all $1\leq i<n-1,$ we conclude that  
		\[|k_1|=|k_{2}|=\ldots=|k_n|.\]
		Therefore, for all $1\leq i<n,$ $T(u_i)=\pm k v_i,\text{ for some } k\neq 0.$ 
		Now, let $T(u_i)=\pm k v_i= z_i$$=\big(z_i(1), z_i(2),\ldots ,z_i(n)\big).$ Since $ u_i$'s are in same facet of $B_{\mathbb{X}}$ then by Theorem \ref{facetTofacet}, $\frac{Tu_i}{\|Tu_i\|}$'s are in same facet of $B_{\mathbb{X}}.$ Therefore, there exists  $1\leq j\leq n$ such that  
		\[z_i(j)= z_i(j)=\ldots =z_i(j)=\pm k.\]
		Let $y\in Ext ~ B_{\mathbb{X}}$ be such that $y$ is in some facet of $B_{\mathbb{X}}$ containing $u_1, u_2,\ldots, u_n.$ Then \[y=\sum_{i=1}^n \alpha_iu_i, ~~\sum_{i=1}^n \alpha_i=1.\]
		Then \[Ty=\sum_{i=1}^n \alpha_iTu_i, ~~\sum_{i=1}^n \alpha_i=1.\]
		Therefore, the absolute value of $j$-th coordinate of $Ty$ is
		\[\Big|\sum_{i=1}^n \alpha_iz_i(j)\Big|=\Big|k\sum_{i=1}^n \alpha_i\Big|=|k|.\]
		Thus, $\|Ty\|=|k|.$ This implies that for all $z\in Ext ~ B_{\mathbb{X}},$ $\|Tz\|=|k|.$ Therefore, from Theorem \ref{isometry}, it follows that $T$ is an isometry multiplied by a constant. This completes the proof of the theorem.
	\end{proof}
	
		Finally, we turn our attention to the spaces where the \textit{Property P}  is not satisfied. First we show that the \textit{Property P} fails to hold in the infinite-dimensional sequence spaces $\ell_p$, where $1 < p < \infty$.
	
	\begin{theorem}\label{ell p}
		Let $\mathbb{X}=\ell_p, 1 < p < \infty $. Then $\mathbb{X}$ does not have the \textit{Property P}.
	\end{theorem}
	\begin{proof}
		We show that for each $0<\epsilon<1,$ there exists an operator $T\in \mathbb{L}(\mathbb{X}),$ which is not a scalar multiple of an isometry but preserves $\epsilon$-orthogonality at each $x\in\mathbb{X}.$ Let $0< \epsilon <1.$ Consider the operator $T\in\mathbb{L}(\mathbb{X})$ given by  \[Tx=\Big(\Big(1-\frac{\epsilon}{p}\Big)x_1,x_2,\dots\Big),\text{ for all } x=(x_1,x_2,\dots)\in\mathbb{X}. \]
		Clearly, $T$ is not a scalar multiple of an isometry.
		Let $x=(x_1,x_2,\dots),y=(y_1,y_2,\dots)\in \mathbb{X}$ be such that $x\perp_By.$  Note that for $u=(u_1,u_2,\dots),v=(v_1,v_2,\dots)\in \mathbb{X},$ the unique semi inner product of $ u $ and $ v $ is given by:
		\[[u,v]_p =  \begin{cases}
			\frac{\sum\limits_{i=1}^{\infty} u_iv_i|v_i|^{p-2}}{\|v\|^{p-2}},&\text{ if }v\neq0\\
			0,&\text{ if }v=0.
		\end{cases}\]
		From \cite[Th. 2]{G67}, it follows that for $u,v\in \mathbb{X},$ $u\perp_B v$ if and only if $[v,u]_p=0.$ Thus,\[x\perp_By\implies[y,x]_p=0\implies \sum\limits_{i=1}^{\infty} y_ix_i|x_i|^{p-2} =0.\] 
		From \cite[pp 308]{CSW17}, it follows that for $u,v\in \mathbb{X},$ $u\perp_B^{\epsilon} v$ if and only if $|[v,u]_p|\leq\epsilon\|u\|\|v\|.$
		Now,
		\begin{eqnarray*}
			|[Ty,Tx]_p|&=&\frac{\Big|\Big(1-\frac{\epsilon}{p}\Big)^py_1x_1|x_1|^{p-2}+\sum\limits_{i=2}^{\infty} y_ix_i|x_i|^{p-2}\Big|}{\|Tx\|^{p-2}}\\ 
			&=&\frac{\Big|-\Big(1-\frac{\epsilon}{p}\Big)^p\sum\limits_{i=2}^{\infty} y_ix_i|x_i|^{p-2}+\sum\limits_{i=2}^{\infty} y_ix_i|x_i|^{p-2}\Big|}{\|Tx\|^{p-2}}\\ 
			&=&\frac{\Big|\Big(1-\big(1-\frac{\epsilon}{p}\big)^p\Big) \sum\limits_{i=2}^{\infty} y_ix_i|x_i|^{p-2}\Big|}{\|Ty\|^{p-2}}\\
			&\leq&\frac{\epsilon \sum\limits_{i=2}^{\infty} |y_i||x_i|^{p-1}}{\|Ty\|^{p-2}}\\
			&\leq& \frac{\epsilon\Big(\sum\limits_{i=2}^{\infty} |y_i|^p\Big)^{\frac{1}{p}}\Big(\sum\limits_{i=2}^{\infty} |x_i|^{p}\Big)^{1-\frac{1}{p}}}{\|Ty\|^{p-2}}\\
			&\leq&\frac{\epsilon\|Ty\| \|Tx\|^{p-1}}{\|Tx\|^{p-2}}\\
			&\leq& \epsilon\|Tx\|\|Ty\|.
		\end{eqnarray*}
		Thus, $Tx\perp_B^{\epsilon}Ty.$
		Therefore, it follows that $T$ preserves $\epsilon$-orthogonality at each $x\in\mathbb{X}.$ This complete the proof.
	\end{proof}
	
	Next, we show that the space $\ell_1$ does not satisfy the \textit{Property P}. 
	
	\begin{theorem}\label{ell 1}
		Let $\mathbb{X}=\ell_1.$ Then $\mathbb{X}$ does not have the \textit{Property P}.
	\end{theorem}
	\begin{proof}
		As before, we show that for each $0<\epsilon<1,$ there exists an operator $T\in \mathbb{L}(\mathbb{X}),$ which is not a scalar multiple of an isometry but preserves $\epsilon$-orthogonality at each $x\in\mathbb{X}.$	Let $0<\epsilon<1.$ Consider the operator $T\in\mathbb{L}(\mathbb{X})$ given by  \[Tx=((1-\epsilon)x_1,x_2,\dots),\text{ for all } x=(x_1,x_2,\dots)\in\mathbb{X}. \]
		Clearly, $T$ is not a scalar multiple of an isometry.
		
		Now we show that $T$ preserves $\epsilon$-orthogonality at each $x\in\mathbb{X}.$ At $x=0,$ $T$ trivially preserves $\epsilon$-orthogonality. Let  $x=(x_1,x_2,\dots)\in \mathbb{X}\setminus\{0\}$ be arbitrary.  Let $u=(u_1,u_2,\dots)\in x^{\perp_B}.$ Then  there exists $ f\in J(x)$ such that $f(u)=0.$ So for each $y=(y_1,y_2,\dots)\in\mathbb{X},~
		f(y)=\sum\limits_{i=1}^{\infty} a_iy_i,$ where
		\[a_i =  
		\begin{cases}
			sgn(x_i),& \text{ if }sgn(x_i)\neq 0\\
			t,\text{ for some } t\in[-1,1],&\text{ if }sgn(x_i)= 0.
		\end{cases}
		\]
		Let $Tx=((1-\epsilon)x_1,x_2,\dots)=(w_1,w_2,\dots)$ and so $sgn(x_i)=sgn(w_i)$ for all $i\in \mathbb{N}.$  This implies that  $f\in J(Tx).$ Now, $f(u)=0$ implies that $\sum\limits_{i=1}^{\infty} a_iu_i=0.$ Hence 
		\begin{eqnarray*}
			|f(Tu)|&=&\Big|(1-\epsilon)a_1u_1+\sum_{i=2}^{\infty} a_iu_i\Big|\\
			&=& |\epsilon\sum_{i=2}^{\infty} a_iu_i|\\
			&\leq& \epsilon\sum_{i=2}^{\infty} |u_i|, \text{ as } |a_i|\leq 1 \text{ for all } i\in \mathbb{N}\\
			&\leq& \epsilon \Big(|(1-\epsilon)u_1|+\sum_{i=2}^{\infty} |u_i|\Big)\\
			&=& \epsilon \|Tu\|.
		\end{eqnarray*} 
		Since $f\in J(Tx),$ it follows from Theorem \ref{charac} that $Tu\in (Tx)^{\perp_B^{\epsilon}}.$ Therefore, $T$ preserves $\epsilon$-orthogonality at $x.$
		This completes the proof of the theorem.
	\end{proof}
	
	In view of Theorem \ref{ell p} and Theorem \ref{ell 1}, it is rather easy to obtain the following corollary.
	\begin{cor}\label{cor ell 1}
		Let $\mathbb{X}$ be the finite-dimensional  Banach space $\ell_p^n,~ 1 \leq p < \infty.$ Then $\mathbb{X}$ does not satisfy the \textit{Property P}.  
	\end{cor}
	
		We note  that the two-dimensional polyhedral Banach space whose unit sphere is the  regular $4$-gon  does not have the \textit{Property P} as it is isometrically isomorphic to $\ell_1^2.$ 
	
	\begin{remark}
		We end this article with the following observation that the \textit{Property P} is not hereditary. 
		\begin{itemize} 
			\item[(i)] Let us consider the space  $\mathbb{X}=\ell_{1}^3$  and the subspace  $\mathbb{Y}=span\{u,v\},$ where $u=(1,0,1),v=(1,1,0)\in\mathbb{X}.$ It is easy to see that the unit sphere $S_{\mathbb {Y}}$  of $\mathbb{Y}$ is a regular hexagon with vertices $\Big\{\pm\Big(\frac{1}{2},0,\frac{1}{2}\Big ), \pm\Big(\frac{1}{2},\frac{1}{2}, 0\Big ), \pm\Big(0,\frac{1}{2},-\frac{1}{2}\Big )\Big\}$  and so by Theorem \ref{regular}, it follows that $\mathbb{Y}$ has the \textit{Property P}. However,  from Corollary \ref{cor ell 1}, it follows that $\mathbb{X}$ lacks the \textit{Property P}. Thus, even if a Banach space lacks the \textit{Property P}, one of its subspaces may possess the same. 
			\item[(ii)] Consider the space $ \mathbb X =\ell_{\infty}^n~(n>2)$.  From Theorem \ref{ell inf}, it follows that $\ell_{\infty}^n~(n>2)$ has the \textit{Property P} whereas the subspace   $\ell_{\infty}^2$ lacks the \textit{Property P}.  So we conclude that if  a Banach space $\mathbb{X}$ has the \textit{Property P} then it is not necessarily true that  a subspace $\mathbb{Y}$ of $\mathbb{X}$ has the \textit{Property P}.
		\end{itemize}
	\end{remark}


\begin{thebibliography}{99}
		
		\bibitem{B35} Birkhoff, G.,  \textit{Orthogonality in linear metric spaces}, Duke Math. J., \textbf{1} (1935) 169-172.
		
		\bibitem{BT06} Blanco, A. and Turn\v{s}ek, A., \textit{On maps that preserve orthogonality in normed spaces}, Proc. Roy. Soc. Edinburgh Sect. A, \textbf{136}(2006) 709-716.
		
		\bibitem{C05} Chemieli\'{n}ski, J., \textit{On an $\epsilon$-Birkhoff orthogonality}, JIPAM. J. Inequal. Pure Appl. Math. \textbf{6} (2005), Article 79, 7 pp.
		
		
		\bibitem{CSW17} Chemieli\'{n}ski, J., Stypula, T. and W\'{o}jcik, P., \textit{Approximate orthogonality in normed linear spaces and its applications}, Linear Algebra Appl., \textbf{531} (2017) 305-317.
		
		\bibitem{DMP22} Dey, S., Mal, A. and  Paul, K., \textit{$k$-smoothness on polyhedral Banach space}, Colloq. Math., \textbf{169} (2022) 25-37.
		
		\bibitem{G67} Giles, J.R., \textit{Classes of semi-inner-product spaces}, Trans. Amer. Math. Soc., \textbf{129} (1967) 436–446.
		
		\bibitem{J47} James, R.C., \textit{Orthogonality and linear functionals in normed linear spaces}, Trans.  Amer.  Math. Soc., \textbf{61} (1947) 265-292.
		
		\bibitem{KS05} Khalil, R. and Saleh, A., \textit{Multi-smooth points of finite order}, Missouri J. Math. Sci., \textbf{17} (2005) 76-87.
		
		\bibitem{K93} Koldobsky, A., \textit{Operators preserving orthogonality are isometries}, Proc. R. Soc. Edinburgh Sect. A, \textbf{123} (1993) 835-837.
		
		\bibitem{LR07} Lin, B. and Rao, T.S.S.R.K., \textit{Multismoothness in Banach Spaces}, Int. J. Math. Math. Sc., \textbf{2007} (2007) Art. ID 52382.
		
		\bibitem{Book24}  Mal,  A., Paul, K. and Sain, D., \textit{Birkhoff-James orthogonality and geometry of operator spaces},  Infosys Science Foundation Series, Springer Singapore, 2024. ISBN 978-981-99-7110-7. https://doi.org/10.1007/978-981-99-7111-4
		
		\bibitem{MDP22}  Mal, A., Dey, S. and Paul, K., \textit{Characterization of k-smoothness of operators defined between infinite-dimensional spaces}, Linear Multilinear Algebra, \textbf{70} (2022) 3477–3489.
		
		\bibitem{MP20} Mal, A. and Paul,  K.,\textit{ Characterization of k-smooth operators between Banach spaces}, Linear Algebra Appl., \textbf{586} (2020) 296-307.
		
		\bibitem{MMPS25} Manna, J., Mandal, K., Paul, K. and Sain, D., \textit{On directional preservation of orthogonality and its application to isometries}, Bull. Sci. Math., \textbf{199} (2025). https://doi.org/10.1016/j.bulsci.2025.103575
		
		\bibitem{PSG16} Paul, K., Sain,  D. and  Ghosh, P., \textit{Birkhoff-James orthogonality and smoothness of bounded linear operators}, Linear Algebra Appl., \textbf{ 506} (2016) 551-563.
		
		\bibitem{RS23} Roy, S. and Sain, D., \textit{Approximate orthogonality in complex normed spaces: a directional approach}, Linear Multilinear Algebra, \textbf{71} (2023) 711–727.
		
		
		\bibitem{S21} Sain, D., \textit{On best approximations to compact operators}, Proc. Amer. Math. Soc. \textbf{149} (2021) 4273–4286.
		
		
		
		\bibitem{SMP24}	Sain,  D.,  Manna, J. and  Paul,  K., \textit{ On local preservation of orthogonality and its application to isometries}, Linear Algebra Appl., \textbf{690} (2024) 112-131.  
		
		\bibitem{SPBB19} Sain, D., Paul, K., Bhunia, P. and Bag, S., \textit{On the numerical index of polyhedral Banach spaces}, Linear Algebra Appl., \textbf{577} (2019) 121-133.
		
		\bibitem{SPMR20} Sain, D., Paul, K., Mal, A. and  Ray, A.,\textit{ A complete characterization of smoothness in the space of bounded linear operators}, Linear Multilinear Algebra, \textbf{68} (2020) 2484-2494.
		
		
		
		\bibitem{W18} W\'{o}jcik, P., \textit{$k$-smoothness: an answer to an open problem}, Math. Scand., \textbf{123} (2018) 85-90.
		
		\bibitem{W22}W\'{o}jcik, P.,\textit{ Approximate orthogonality in normed spaces and its applications II}, Linear Algebra Appl., \textbf{632} (2022) 258–267.
	\end{thebibliography}
\end{document}